\documentclass[11pt,letterpaper]{amsart}
\usepackage[usenames,dvipsnames]{color}
\usepackage{amsthm,amsfonts,amssymb,amsmath,amsxtra}
\usepackage[all]{xy}
\SelectTips{cm}{}
\usepackage{xr-hyper}
\usepackage[pdfpagelabels,pdftex,hidelinks]{hyperref}
\hypersetup{
  colorlinks   = true, %Colours links
  urlcolor     = blue, %Colour for external hyperlinks
  linkcolor    = blue, %Colour of internal links
  citecolor   = blue, %Colour of citations
pdftitle={},
  pdfauthor={},
  pdfsubject={},
  pdfkeywords={},
  breaklinks=true,
  bookmarksopen=true,
  bookmarksnumbered=true,
  pdfpagemode=UseOutlines,
  plainpages=false}
\usepackage{cleveref}

\usepackage{verbatim}
\usepackage{caption}

\usepackage{lineno}
\usepackage{mathrsfs}

\RequirePackage{xspace}
\RequirePackage{etoolbox}
\RequirePackage{varwidth}
\RequirePackage{enumitem}
\RequirePackage{mathtools}
\RequirePackage{longtable}
\RequirePackage{multirow}

\def\ge{\geqslant}
\def\le{\leqslant}
\def\a{\alpha}
\def\b{\beta}
\def\g{\gamma}
\def\G{\Gamma}
\def\d{\delta}
\def\D{\Delta}
\def\L{\Lambda}
\def\e{\epsilon}

\def\t{\tau}
\def\th{\theta}
\def\k{\kappa}
\def\l{\lambda}

\def\i{^{-1}}

\def\<{\langle}
\def\>{\rangle}

\newcommand{\BA}{\ensuremath{\mathbb {A}}\xspace}

\newcommand{\BF}{\ensuremath{\mathbb {F}}\xspace}
\newcommand{{\BG}}{\ensuremath{\mathbb {G}}\xspace}
\newcommand{\BH}{\ensuremath{\mathbb {H}}\xspace}
\newcommand{\BI}{\ensuremath{\mathbb {I}}\xspace}

\newcommand{{\BK}}{\ensuremath{\mathbb {K}}\xspace}
\newcommand{\BL}{\ensuremath{\mathbb {L}}\xspace}

\newcommand{\BQ}{\ensuremath{\mathbb {Q}}\xspace}
\newcommand{\BR}{\ensuremath{\mathbb {R}}\xspace}

\newcommand{\BT}{\ensuremath{\mathbb {T}}\xspace}

\newcommand{\BZ}{\ensuremath{\mathbb {Z}}\xspace}

\newcommand{\CC}{\ensuremath{\mathcal {C}}\xspace}

\newcommand{\CE}{\ensuremath{\mathcal {E}}\xspace}

\newcommand{\CG}{\ensuremath{\mathcal {G}}\xspace}
\newcommand{\CH}{\ensuremath{\mathcal {H}}\xspace}
\newcommand{\CI}{\ensuremath{\mathcal {I}}\xspace}

\newcommand{\CK}{\ensuremath{\mathcal {K}}\xspace}
\newcommand{\CL}{\ensuremath{\mathcal {L}}\xspace}

\newcommand{\CO}{\ensuremath{\mathcal {O}}\xspace}

\newcommand{\CR}{\ensuremath{\mathcal {R}}\xspace}

\newcommand{\CT}{\ensuremath{\mathcal {T}}\xspace}

\newcommand{\Ad}{{\mathrm{Ad}}}

\newcommand{\GL}{\mathrm{GL}}

\DeclareMathOperator{\Hom}{Hom}

\let\Im\relax
\DeclareMathOperator{\Im}{Im}

\newcommand{\red}{\ensuremath{\mathrm{red}}\xspace}

\newcommand{\Sp}{{\mathrm{Sp}}}

\DeclareMathOperator{\tr}{tr}

\newcommand{\ov}{\overline}

\def\tw{\tilde w}

\def\pr{{\rm pr}}
\def\tPhi{\widetilde \Phi}

\def\der{{\rm der}}
\def\ov{\overline}

\def\sc{{\rm sc}}

\def\unip{{\rm unip}}

\def\tG{{\tilde G}}
\def\tT{{\tilde T}}

\def\tG{{\tilde G}}
\def\tB{{\tilde B}}
\def\tphi{{\tilde \phi}}
\def\tth{{\tilde \theta}}
\def\tiota{{\tilde \iota}}
\def\tCI{{\tilde\CI}}
\def\tCO{{\tilde\CO}}
\def\tCC{{\tilde\CC}}
\def\te{{\tilde \epsilon}}

\def\tX{{\tilde X}}

\def\tSigma{{\tilde \Sigma}}
\def\hCO{{\hat \CO}}
\def\hSigma{{\hat \Sigma}}
\def\hpsi{{\hat \psi}}
\def\hXi{{\hat \Xi}}
\def\tBK{{\tilde{\mathbb K}}}
\def\tBL{{\tilde{\mathbb L}}}
\def\tBT{{\tilde{\mathbb T}}}
\def\tL{{\tilde L}}

\newtheorem{theorem}{Theorem}
\newtheorem{proposition}[theorem]{Proposition}
\newtheorem{lemma}[theorem]{Lemma}

\newtheorem{corollary}[theorem]{Corollary}

\theoremstyle{definition}

\numberwithin{equation}{section}
\numberwithin{theorem}{section}

\setitemize[0]{leftmargin=*,itemsep=\the\smallskipamount}
\setenumerate[0]{leftmargin=*,itemsep=\the\smallskipamount}

\renewcommand{\to}{%
   \ifbool{@display}{\longrightarrow}{\rightarrow}%
   }
\let\shortmapsto\mapsto
\renewcommand{\mapsto}{%
   \ifbool{@display}{\longmapsto}{\shortmapsto}%
   }
\newlength{\olen}
\newlength{\ulen}
\newlength{\xlen}
\newcommand{\xra}[2][]{%
   \ifbool{@display}%
      {\settowidth{\olen}{$\overset{#2}{\longrightarrow}$}%
       \settowidth{\ulen}{$\underset{#1}{\longrightarrow}$}%
       \settowidth{\xlen}{$\xrightarrow[#1]{#2}$}%
       \ifdimgreater{\olen}{\xlen}%
          {\underset{#1}{\overset{#2}{\longrightarrow}}}%
          {\ifdimgreater{\ulen}{\xlen}%
             {\underset{#1}{\overset{#2}{\longrightarrow}}}
             {\xrightarrow[#1]{#2}}}}%
      {\xrightarrow[#1]{#2}}
   }
\makeatother
\newcommand{\xyra}[2][]{%
   \settowidth{\xlen}{$\xrightarrow[#1]{#2}$}%
   \ifbool{@display}%
      {\settowidth{\olen}{$\overset{#2}{\longrightarrow}$}%
       \settowidth{\ulen}{$\underset{#1}{\longrightarrow}$}%
       \ifdimgreater{\olen}{\xlen}%
          {\mathrel{\xymatrix@M=.12ex@C=3.2ex{\ar[r]^-{#2}_-{#1} &}}}%
          {\ifdimgreater{\ulen}{\xlen}%
             {\mathrel{\xymatrix@M=.12ex@C=3.2ex{\ar[r]^-{#2}_-{#1} &}}}
             {\mathrel{\xymatrix@M=.12ex@C=\the\xlen{\ar[r]^-{#2}_-{#1} &}}}}}%
      {\mathrel{\xymatrix@M=.12ex@C=\the\xlen{\ar[r]^-{#2}_-{#1} &}}}%
   }
\makeatletter
\newcommand{\xla}[2][]{%
   \ifbool{@display}%
      {\settowidth{\olen}{$\overset{#2}{\longleftarrow}$}%
       \settowidth{\ulen}{$\underset{#1}{\longleftarrow}$}%
       \settowidth{\xlen}{$\xleftarrow[#1]{#2}$}%
       \ifdimgreater{\olen}{\xlen}%
          {\underset{#1}{\overset{#2}{\longleftarrow}}}%
          {\ifdimgreater{\ulen}{\xlen}%
             {\underset{#1}{\overset{#2}{\longleftarrow}}}
             {\xleftarrow[#1]{#2}}}}%
      {\xleftarrow[#1]{#2}}
   }
\newcommand{\isoarrow}{%
   \ifbool{@display}{\overset{\sim}{\longrightarrow}}{\xrightarrow\sim}%
   }

\newcommand{\sm}{{\,\smallsetminus\,}}

\usepackage{upgreek}
\makeatletter
\newcommand{\colim@}[2]{%
  \vtop{\m@th\ialign{##\cr
    \hfil$#1\operator@font lim$\hfil\cr
    \noalign{\nointerlineskip\kern1.5\ex@}#2\cr
    \noalign{\nointerlineskip\kern-\ex@}\cr}}%
}
\newcommand{\colim}{%
  \mathop{\mathpalette\colim@{\rightarrowfill@\textstyle}}\nmlimits@
}
\newcommand{\prolim@}[2]{%
  \vtop{\m@th\ialign{##\cr
    \hfil$#1\operator@font lim$\hfil\cr
    \noalign{\nointerlineskip\kern1.5\ex@}#2\cr
    \noalign{\nointerlineskip\kern-\ex@}\cr}}%
}
\newcommand{\prolim}{%
  \mathop{\mathpalette\colim@{\leftarrowfill@\textstyle}}\nmlimits@
}
\makeatother

\begin{document}
\title[]{Symmetric spaces over a finite ring} 

\author[Ben Liu]{Ben Liu}
\address{Academy of Mathematics and Systems Science, Chinese Academy of Sciences, Beijing 100190, China}
\email{liubenmath@gmail.com}

\author[Sian Nie]{Sian Nie}
\address{Academy of Mathematics and Systems Science, Chinese Academy of Sciences, Beijing 100190, China}

\address{School of Mathematical Sciences, University of Chinese Academy of Sciences, Chinese Academy of Sciences, Beijing 100049, China}
\email{niesian@amss.ac.cn}

\begin{abstract}
Deep level Deligne-Lusztig representations, as natural extensions of the classical Deligne-Lusztig representations, have recently seen various applications in the Langlands program. In this note, we address their distinction problem for symmetric pairs over finite fields, yielding a generalization of a classical result by Lusztig.
\end{abstract}

\maketitle

\section{Introduction}
In \cite{Lu79} and \cite{Lu04}, Lusztig introduced the notion of deep level Deligne-Lusztig representations for Lie groups over finite rings. Such objects extends the classical construction of Deligne and Lusztig \cite{DL76} in the finite field setting, and they are expected to provide a geometric interpretation of supercuspidal representations of $p$-adic fields. In the past two decades, the study of deep level Deligne-Lusztig representations has been undergoing rapid development. We refer to \cite{BW16}, \cite{Chan20}, \cite{CI23}, \cite{CS23}, \cite{Iva23} \cite{CO25a}, \cite{BC24}, \cite{Chan24}, \cite{Nie24}, \cite{IN25}, \cite{CO25b}, \cite{INY25}, \cite{Tak26} and references therein for their remarkable properties and applications in Langlands program. 

\subsection{The main result}   
In this paper we study the distinction problem on deep level Deligne-Lusztig representations for symmetric pairs. More precisely, let $G$ be a connected reductive group with an involution defined over $\BF_q$. Consider an involution $\th$ of $G$ defined over $\BF_q$ and denote by $G^\th$ the $\th$-fixed point subgroup of $G$. 

For $r \in \BZ_{\ge 0}$ let $\CO_r = \BF_q[t]/(t^{r+1})$ be a finite ring over $\BF_q$. Following \cite{Nie24}, to each $\BF_q$-rational maximal torus $T \subseteq G$ and a character $\phi: T(\CO_r) \to \ov\BQ_\ell^\times$ (with $\ell$ a fixed prime not dividing $q$), one can associate a virtual representation $\CR_{T_r}^{G_r}(\phi)$ of $G(\CO_r)$, which is referred to a deep level Deligne-Lusztig representation. The main purpose of this paper is to compute the multiplicity of the trivial representation of $G^\th(\CO_r)$ in $\CR_{T_r}^{G_r}(\phi)$: \[\<\CR_{T_r}^{G_r}(\phi), 1\>_{G^\th(\CO_r)} = \frac{1}{|G^\th(\CO_r)|} \sum_{g \in G^\th(\CO_r)} \tr(g; \CR_{T_r}^{G_r}(\phi)).\] To this end, we consider the set $\Theta_T$ of elements $w \in G(\BF_q)$ such that the conjugate involution $\th_w := \Ad(w) \th \Ad(w)\i$ fixes $T$, and associated to each $w \in \Theta_T$ a quadratic character $\e_w: T^{\th_w}(\CO_r) \to \{\pm 1\}$ as in \S\ref{sec:proof-main}. Let $\Theta_\phi$ be the set of elements $w \in \Theta_T$ such that $\phi|_{T^{\th_w}(\CO_r)} = \e_{\phi, w}$.

The main result is as follows.
\begin{theorem} [cf. Theorem \ref{formula}] \label{into-main}
    We have \[\<\CR_{T_r}^{G_r}(\phi), 1\>_{G^\th(\CO_r)} = \sum_{w \in T(\BF_q) \backslash \Theta_\phi / G^\th(\BF_q)} (-1)^{n_{\phi, w}},\] where the integers $n_{\phi, w}$ are defined using the $\th_w$-action on the Howe filtration of $\phi$, see \S \ref{sec:proof-main}.
\end{theorem}

When $T$ is elliptic, it is proved in \cite{Nie24} that original deep level Deligne-Lusztig representations $R_{T_r}^{G_r}(\phi)$ constructed by Lusztig coincide with the representations $\CR_{T_r}^{G_r}(\phi)$ used in this paper. Hence Theorem \ref{into-main}  also provides a multiplicity formula for $R_{T_r}^{G_r}(\phi)$ in the elliptic case. 

\smallskip

When $r= 0$, that is, $\CR_{T_r}^{G_r}(\phi)$ is a classical Deligne-Lusztig representation, Theorem \ref{into-main} is due to Lusztig \cite{Lu90}. Our proof follows the general strategy of Lusztig. The major innovation in our computation is the reduction procedure from positive depths to the classical zero depth. A remarkable new phenomenon  is that the contributions of positive depths to the multiplicity are encoded in the complexes  \[R\G_c(\BG_a, f^*\CL),\] where $\CL$ is an Artin-Schreier sheaf and $f: \BG_a \to \BG_a, \ z \mapsto z^2$ is the square map. By coincidence, this complex happens to be concentrated at degree one with dimension one. This key property plays an essential role in establishing the multiplicity formula. 

\smallskip

It is an interesting problem to consider the more general case that $G(\CO_r) = \CG(\CO_r)$, where $\CG$ is a parahoric group model over the integer ring $\CO$ of a $p$-adic field. We remark that our approach still works for the symmetric pairs $(\CG, \CG^\th)$ such as $(\GL_n, \GL_m \times \GL_{n-m})$, $(\GL_{2n}, \Sp_{2n})$ and $(\CG \times \CG, \CG)$. 

\smallskip

In \cite{HM}, Hakim and Murnaghan studied the distinction problem on supercuspidal representations of $p$-adic groups for symmetric pairs in the framework of Yu \cite{Yu}. On the other hand,  supercuspidal representations are closely related to deep level Deligne-Lusztig representations (see \cite{Nie24} and \cite{CO25b}). We hope the method of this paper would provide a geometric interpretation of the work \cite{HM}.

\subsection{Structure of the paper}
The paper is organized as follows. In \S\ref{sec:pre}, we collect basic definitions and properties on the truncated Moy-Prasad algebraic group $G_r$ and its symmetric quotient space $G_r/G_r^\th$. In \S\ref{sec:pDL}, we recall the construction of deep Deligne-Lusztig representations $\CR_{T_r}^{G_r}(\phi)$ and a trace formula  using thr Green functions. In \S\ref{sec:proof-main} we present the multiplicity formula and outline its proof. We show that the multiplicity formula on $\CR_{T_r}^{G_r}(\phi)$ follows from that on the representation $\CR_{\tT_r}^{\tG_r}(\tphi)$. Here $\tG_r \to G_r$ is a group isogeny 
equipped with a lifting involution $\tth$ of $\th$ such that $\tG_r^\tth$ is connected, and $\CR_{\tT_r}^{\tG_r}(\tphi)$ is the deep level Deligne-Lusztig representation associated to the natural lift $(\tT, \tphi)$ of the pair $(T, \phi)$. To handle the representation $\CR_{\tT_r}^{\tG_r}(\tphi)$, we use connectedness of $\tG_r^\tth$ to express the multiplicity in terms of cohomology of certain varieties $\tSigma_w$. In \S\ref{sec:vanish} and \S\ref{sec:first-reduction} we reduce the cohomology computation of $\tSigma_w$ to that of a much simpler variety $\Sigma_w^\BK$. In \ref{sec:coho-BK} and \ref{sec:proof-tilde}, we compute the cohomology of $\Sigma_w^\BK$ and finish the proof of the multiplicity formula on $\CR_{\tT_r}^{\tG_r}(\tphi)$.

\section{The symmetric quotient} \label{sec:pre}

Let $\BF_q$ be a finite field with $q$ odd. Denote by $\ov\BF_q[[\varpi]]$ the ring of formal power series over $\ov\BF_q$. Let $F$ be the Frobenius automorphism of $\ov\BF_q[[\varpi]]$ given by $\sum_i a_i \varpi^i \mapsto \sum_i a_i^q  \varpi^i$.

\subsection{Positive loop group}\label{sec:setup}
Let $G$ be a reductive group over $\BF_q$, whose Frobenius automorphism is still denoted by $F$. Put $\widetilde\BR_{\ge 0} = \BR_{\ge 0} \sqcup \{s+; s \in \BR_{\ge 0}\}$. For $r \in \widetilde\BR_{\ge 0}$ let $\CG^r$ (resp. $\CG^{r+}$) be the subgroup \[G(1 + \varpi^k \ov\BF_q[[\varpi]]) \subseteq G(\ov\BF_q[[\varpi]]),\] where $k = \min\{i \in \BZ; i \ge r\}$ (resp. $k = \min\{i \in \BZ; i >r\}$). For any $0 \le s \le r \in \widetilde\BR$ we put \[G_{s:r} = \CG^s / \CG^{r+},\] which is a smooth affine group scheme over $\ov\BF_q$. Note that $F$ acts naturally on $G_{s:r}$, equipping it with an $\BF_q$-rational structure. We put $G_r = G_{0:r} \cong G_{0+:r} \rtimes G_0$ with $G_0 = G(\ov\BF_q)$.

For a closed subgroup $H \subseteq G$ defined over $\ov\BF_q$, let  $H_{s:r} \subseteq H_r$ be the closed subgroups of $G_r$ defined in a similar way.

\subsection{Root System}\label{subsec:root}
Let $T \subseteq G$ be a fixed maximal torus over $\BF_q$. We write $W = N_G(T)/T$ for the Weyl group of $T$, $\Phi = \Phi(G, T)$ for the root system of $T$ in $G$, $G^\a \subseteq G$ for the root subgroup of $\a$, and $\tilde\Phi = \Phi \times \BZ \,\sqcup\, \BZ$ for the set of affine roots. For $\alpha \in \Phi$, we write $\alpha^\vee \colon \BG_m \to T$ for the corresponding coroot and $u_\a \colon \BG_a \to G^\a$ for a fixed $\ov\BF_q$-rational parametrization. We fix a basis $(\l_i)_{i=1}^m$ of the coweight lattice $X_*(T)$. Let $f \in \tPhi$. We write $\a_f \in \Phi \sqcup \{0\}$, $n_f \in \BZ$ such that $f = (\a_f,n_f) = \a_f + n_f$. If $\a_f = 0$ and $n_f \ge 1$, we define  \[u_f: \BA^m \to \CG, \quad (x_i)_i \mapsto \prod_i \l_i(1 + \varpi^{n_f} [x_i]),\] where $[x_i]$ denotes the Teichm\"{u}ller lift of $x_i \in \ov\BF_q$. If $\a_f \in \Phi$, we define \[u_f: \BG_a \to \CG, \quad x \mapsto u_{\a_f}(\varpi^{n_f} [x]).\] 

We fix an $\ov\BF_q$-rational Borel subgroup $B = T U \subseteq G$ with $U$ the unipotent radical. Denote by $\Phi_B \subseteq \Phi$ the set of roots in $B$. 

We define a linear order on $\tPhi_r$ such that $f < f'$ if either $n_f < n_{f'}$ or $n_f = n_{f'}$ and $\a_{f'} - \a_f$ is a sum of roots in $\Phi_B$. Let $\tPhi^{>0} = \{f \in \tPhi; f > 0\}$, which consists of affine roots $f \in \tPhi_r$ such that either $0 < \a_f$ or $\a_f = 0$ and $\a_f \in \Phi_B$.  

Let $r$ be a fixed non-negative integer. Let $\CI_r^+ \subseteq \CI_r$ be the inverse images of $U_0 \subseteq B_0$ respectively under the natural reduction map $G_r \to G_0$. We put $\tPhi_r^{>0} = \{f \in \tPhi^{> 0}; , n_f \le r\}$. Then $\CI_r^+ = \prod_{f \in \tPhi_r^{>0}} G_r^f$. Here $G_r^f$ is the natural image of $\CG^f$ in $G_r$.

\subsection{Involution of $G_r$}
Let $\th$ be an involution of $G_r$. For $x \in G_r$ we put \[x \cdot \th = \th_x = \Ad(x) \th \Ad(x\i)\] which is also an involution of $G_r$.

\begin{lemma} \label{solution}
    Let $A = \th(A) \subseteq G_{0+:r}$ be a subgroup such that either $A$ is finite or $A = G_{0+:r}$ or $A = T_{0+,r}$. Let $x \in A$ such that $x \th(x) = 1$. Then there exists $y \in A$ such that $x = y \th(y)\i$.
\end{lemma}
\begin{proof}
    As $q$ is odd, the map $x \mapsto x^2$ induces a bijection from $A$ to itself. Then the statement follows from \cite[Lemma2.1]{HM}.
\end{proof}

For $t \in T_0^F = T(\BF_q)$ we denote by $Z_G^\circ(t)$ the identity component of the centralizer $Z_G(t)$ of $t$ in $G$.
\begin{lemma} \label{stable}
    Let $\iota$ be an involution of $G$ over $\BF_q$, $\th \in G_r^F \cdot \iota$ and $t \in T_0^F$ such that $\th(t) = t$. Then there exists an involution $\iota'$ of $Z_G^\circ(t)$ over $\BF_q$ such that the restriction of $\th$ to $Z_G^\circ(t)_r$ belongs to $Z_G^\circ(t)_{0+:r}^F \cdot \iota'$.
\end{lemma}
\begin{proof}
    By definition, $\th = \iota_x$ for some $x \in G_r^F$. Write $x = \d w$ with $w \in G_0 = G(\ov\BF_q)$ and $\d \in G_{0+:r}$. As $x$ is fixed by $F$, $w$ and $\d$ are also fixed by $F$. Let $\iota' = w \cdot \iota = \iota_w$ and $y = \d\iota'(\d\i) \in G_{0+:r}^F$. Then $\th = \Ad(y) \iota'$ fixes $t \in T(\BF_q)$. This implies that $\iota'(t) = t$ and $y \in Z_G(t) \cap G_{0+:r}^F = Z_G^\circ(t)_{0+:r}^F$. Moreover, since $y\iota'(y) = 1$, by Lemma \ref{solution} there exists $z \in Z_G^\circ(t)_{0+:r}^F$ such that $y = z \iota'(z\i)$. Thus $\th = \Ad(y) \iota' = (\iota')_z \in Z_G^\circ(t)_{0+:r}^F \cdot \iota'$ as desired.
\end{proof}

\subsection{The group $\tG_r$}
Let $\iota$ be an involution of $G$ defined over $\BF_q$. We fix an involution $\th$ in $G_r^F \cdot \iota$. 

Let $\pi: G_\sc \to G$ be the simply connected cover of the derived subgroup of $G$ together with a maximal torus $T_\sc = \pi\i(T)$. Let $\iota_\sc$ be a lift of $\iota$. Consider the following algebraic group \[\tG_r = G_{0+:r} \rtimes (G_\sc)_0,\] with an involution of $\tiota$ given by $(h, s) \mapsto (\iota(h), \iota_\sc(s))$. There is a natural projection $\tG_r \to G_r$ given by $(h, s) \mapsto h \pi(s)$, which we still denote by $\pi$.

Applying Lemma \ref{stable} (with $t=1$), we may assume that $\th = \d \cdot \iota$ for some $\d \in G_{0+:r}^F$. Then we put $\tth = \d \cdot \tiota$, which is an involution of $\tG_r$. Note that $\pi \tth = \th \pi$. Let $\tT_r = \pi\i(T_r) = T_{0+:r} \rtimes (T_\sc)_0$. We define \begin{align*}
    \Theta_{T_r} := \Theta_{T_r, \th} &= \{x \in G_r; \th_x(T_r) = T_r\} \\ \Theta_{\tT_r} := \Theta_{\tT_r, \tth} &= \{x \in \tG_r; \tth_x(\tT_r) = \tT_r\}.
\end{align*} Note that $\Theta_{T_r} = T_r \Theta_{T_r} G_r^\th$ and $\Theta_{\tT_r} = \tT_r \Theta_{\tT_r} \tG_r^\tth$.
\begin{lemma} \label{Theta}
    We have $\Theta_{\tT_r} = \pi\i(\Theta_{T_r})$.
\end{lemma}
\begin{proof}
    Note that $\pi \tth_x = \th_{\pi(x)} \pi$ for any $x \in \tG_r$ and $T_r = \pi(\tT_r) Z^\circ_r$. Here $Z^\circ$ denotes the identity component of the center of $G$.  Let $x \in \Theta_{\tT_r}$. Then we have \[\th_{\pi(x)}(T_r) = \th_{\pi(x)}(\pi(\tT_r)) \th_{\pi(x)}(Z^\circ_r) = \pi(\tth_x(\tT_r)) Z^\circ_r = \pi(\tT_r) Z^\circ_r = T_r,\] which implies that $\pi(x) \in \Theta_{T_r}$ and hence $\pi(\Theta_{\tT_r}) \subseteq \Theta_{T_r}$. On the other hand, as $\tT_r = \pi\i(T_r)$, it follows by definition that $\pi\i(\Theta_{T_r}) \subseteq \Theta_{\tT_r}$. Thus $\pi\i(\Theta_{T_r}) = \Theta_{\tT_r}$ as desired.  
\end{proof}

\begin{lemma} \label{Theta}
    Let $w, w' \in \Theta_{\tT_r}$ such that $w'w\i \in G_{0+:r}$. Then we have $w' \in T_{0+:r} w G_{0+:r}^\tth$.
\end{lemma}
\begin{proof}
    Assume $w' = \d w$ for some $\d \in G_{0+:r}$. Then $\tth_{w'} = \Ad(\d) \tth_w \Ad(\d)\i = \Ad(\d \tth_w(\d)\i) \tth_w$. Hence $\tT_r$ is normalized by $\d \tth_w(\d)\i \in G_{0+:r}$, which implies that $\d \tth_w(\d)\i \in T_{0+:r}$. By Lemma \ref{solution} there exists $t \in T_{0+:r}$ such that $\d \tth_w(\d)\i = t\i \tth_w(t)$,  that is, $t\d \in G_{0+:r}^{\tth_w} = w G_{0+:r}^\tth w\i$. Then we have $w' = t\i (t\d) w \in t\i (w G_{0+:r}^\tth w\i) w \subseteq T_{0+:r} w G_{0+:r}^\tth$ as desired.
\end{proof}

We put $\tG_0 = (G_\sc)_0 = G_\sc(\ov\BF_q)$ and $\tT_0 = (T_\sc)_0 = T_\sc(\ov\BF_q)$. Let $\tilde p_0: \tG_r \to \tG_0 \cong \tG_r/G_{0+:r}$ be the natural projection. We set $\Theta_{\tT_0} = \Theta_{\tT_0, \tiota} = \{x \in \tG_0; \tiota_x(\tT_0) = \tT_0\}$.
\begin{proposition} \label{level-zero}
    We have that \[\tG_0 / \tG_0^\tiota = \sqcup_{w_0 \in \tT_0 \backslash \Theta_{\tT_0} / G_0^\tiota} \tB_0 w_0 \tG_0^\tiota/\tG_0^\tiota.\] Moreover, for any $w_0, w_0' \in \Theta_{\tT_0}$ such that $w_0' \in \tB_0 w_0 \tG_0^\tiota$ we have $w_0' \in \tT_0 w_0 \tG_0^\tiota$.
\end{proposition}

For $w \in \Theta_{\tT_r}$ we define $\tCO_w = \tT_0 w \tG_r^\tth/\tG_r^\tth$ and $\tCC_w = \tCI_r \tCO_w$. 
\begin{lemma} \label{decomposition}
    There is a decomposition \[\tG_r / \tG_r^\tth = \sqcup_{w \in \tT_r \backslash \Theta_{\tT_r} / \tG_r^\tth} \tCC_w.\] 
\end{lemma}
\begin{proof}
     We denote by $\pi_r: (G_\sc)_r \to G_r$ the natural projection. Let $Z^\circ$ be the identity component of the center of $G$. For $x \in \Theta_{\tT_0}$ we have $\tiota_x(T_\sc) = T_\sc$ and hence \[\tiota_x(T_{0+:r}) = \iota_{\pi(x)}(\pi_r((T_\sc)_{0+:r}) Z^\circ_{0+:r}) = \pi_r(\iota_\sc)_x((T_\sc)_{0+:r}) Z^\circ_{0+:r} = \pi_r((T_\sc)_{0+:r}) Z^\circ_{0+:r} = T_{0+:r}.\] In particular, $\tiota_x(\tT_r) = \tiota_x(\tT_0 T_{0+:r}) = \tT_r$. As $\tth = \d \cdot \tiota$ for some $\d \in G_{0+:r}^F$, we have $\Theta_{\tT_0} \d\i \subseteq \Theta_{\tT_r}$. As $G_{0+:r} \subseteq \tCI_r$, it follows from Lemma \ref{solution} and Proposition \ref{level-zero} that \[\tG_r / \tG_r^\tth = \cup_{w_0 \in \Theta_{\tT_0}} \tCI_r w_0 \tG_r^\tth/\tG_r^\tth =  \cup_{w \in \Theta_{\tT_r}} \tCI_r w \tG_r^\tth/\tG_r^\tth = \cup_{w \in \Theta_{\tT_r}} \tCC_w.\] Let $w, w' \in \Theta_{\tT_r}$ such that $\tCC_w = \tCC_{w'}$. Then $\tilde p_0(w') \in \tB_0 \tilde p_0(w) \tG_0^{\tiota}$. By Proposition \ref{level-zero} we have $\tilde p_0(w') \in \tT_0 \tilde p_0(w) \tG_0^{\tiota}$. Note that $\tilde p_0(\tG^\tth) = \tG_0^\tiota$ by Lemma \ref{solution}. It follows that $w' \in G_{0+:r} \tT_r w \tG_r^\tth$. By Lemma \ref{Theta}, we have $w' \in  \tT_r w \tG_r^\tth$ and hence $\tG_r / \tG_r^\tth = \sqcup_{w \in \tT_r \backslash \Theta_{\tT_r} / \tG_r^\tth} \tCC_w$ as desired.   
\end{proof}

\subsection{Parametrizations} \label{subsec:para}
We view $\CI_r^+ \subseteq U_r G_{0+:r}$ as a subset of $\tG_r$ in the natural way. As $q$ odd, the map $x \mapsto x^2$ induces a variety automorphism of $\CI_r^+$. We denote by $x \mapsto x^{\frac{1}{2}}$ the inverse morphism. Let $w \in \Theta_{\tT_r}$. Then $\tth_w$ acts on $\tPhi_r$ in a natural way.

We fix bases $(\mu_i)_{i=1}^{m_1}$ and $(\nu_i)_{i=1}^{m_2}$ of the lattices $X_*(T)^{-\tth_w}$ and $X_*(T)^{\tth_w}$ respectively. Let $f \in \tPhi^{>0}$. If $f \in \BZ_{\ge 1}$ we put $\BA_{f, w} = \BA^{m_1}$, $\BA_f^w = \BA^{m_2}$ and \begin{align*} u_{f, w}: \BA_{f, w} \to \tG_r, &\quad (x_i)_i \to \prod_i \mu_i(1 + \varpi^f x_i) \\ u_f^w: \BA_f^w \to \tG_r^{\tth_w}, &\quad (x_i)_i \mapsto \prod_i \nu_i(1 + \varpi^f x_i).\end{align*} If $f \in \Phi \times \BZ$, we put $\BA_{f, w} = \BA_f^w = \BG_a$, $u_{f, w} = u_f$ as in \S\ref{subsec:root} and \[u_f^w: \BA_f^w \to \tG_r^{\tth_w}, \quad x \mapsto u_f(x) {}^{\tth_w} u_f(x) [{}^{\tth_w} u_f(x)\i, u_f(x)\i]^{\frac{1}{2}}.\] Here ${}^{\tth_w} z = \tth_w(z)$ and $[x, y] = x y x\i y\i$ is the commutator of $x$ and $y$. 

\

Let $E \subseteq \tPhi_r^{>0}$ be a subset. We set \begin{align*} A_{E, w} &= \{f \in E; \tth_w(f) \le f\} \\ B_{E, w} &= \{f \in A_{E, w}; \tth_w(f) \in E\}.\end{align*} Let $C_{E, w}$ (resp. $D_{E, w}$) be the set of affine roots $f = \tth_w(f) \in E \sm \BZ$ such that $\tth_w$ acts on $\CG^f$ trivially (resp. non-trivially). We define \[\BA_{E, w} = \prod_{f \in A_{E, w} \sm C_{E, w}} \BA_{f, w}, \quad \BA_E^w = \prod_{f \in B_{E, w} \sm D_{E, w}} \BA_w^f.\] 

We say $E \subseteq \tPhi_r^{>0}$ is $\tth_w$-admissible if the following conditions
\begin{itemize} 
    \item $\{1, 2, \dots, r\} \subseteq E$;

    \item if $f, f' \in E$ satisfies $f+f' \in \tPhi_r^{>0}$, then $f + f' \in E$;

    \item if $f \in E$ satisfies $f < \tth_w(f) \in \tPhi_r^{>0}$, then $\tth_w(f) \in E$.
\end{itemize} In particular, the product $G_r^E = \prod_{f \in E} G_r^f \subseteq \CI_r^\dag$ is subgroup if $E$ is $\tth_w$-admissible.

For $A \subseteq \tG_r$ we set $A^{\tth_w} = \{x \in A; \tth_w(x) = x\}$.
\begin{lemma} \label{parameter}
   Let $E \subseteq \tPhi_r^{>0}$ be $\tth_w$-admissible. Then the maps $(x_f)_f \mapsto \prod_f u_f^w(x_f)$ and $(x_f)_f \mapsto \prod_f u_{f, w}(x_f)$ induce isomorphisms $u_E^w: \BA_E^w \cong (G_r^E)^{\tth_w}$ and $ u_{E, w}: \BA_{E, w} \cong G_r^E / (G_r^E)^{\tth_w}$ respectively.  
\end{lemma}
\begin{proof}
    For $i \in \BZ_{\ge 0}$ we put $E^i = \{f \in E; n_f = i\}$ and $E_i = \{f \in E; n_f \le i\}$. By the assumption on $E$ we have $\Im u_f^w \subseteq \CI_{E, r}^{\tth_w}$ for $f \in B_{E, w}$.
    
    We argue by induction on $r$. If $r = 0$, then $G_r^E \subseteq U_0$ is a subgroup normalized by $T_0$ and the statement is a classical result. Suppose $r \ge 1$ and the statement holds for $r-1$. Consider the exact sequence \[1 \to G^E_{r:r} \to G^E_r\to G^{E_{r-1}}_{r-1} \to 1,\] where $G^E_{r:r}$ is the kernel of the natural projection $G^E_r \to G^{E_{r-1}}_{r-1}$. By induction hypothesis, the map $u_{E_{r-1}}^w: \BA_{E_{r-1}}^w \overset \sim \to (G_{r-1}^{E_{r-1}})^{\tth_w}$ is an isomorphism which factors through the natural maps $(G_r^E)^{\tth_w} \to (G_{r-1}^{E_{r-1}})^{\tth_w}$. Thus the natural map $(G^E_r)^{\tth_w} \to (G^{E_{r-1}}_{r-1})^{\tth_w}$ is surjective, and hence we have two exact sequences \begin{gather*} 1 \to (G^E_{r:r})^{\tth_w} \to (G^E_r)^{\tth_w} \to (G^{E_{r-1}}_{r-1})^{\tth_w} \to 1; \\  1 \to G_{r:r}^E / (G^E_{r:r})^{\tth_w} \to G_r^E / (G^E_r)^{\tth_w} \to G^{E_{r-1}}_{r-1} / (G^{E_{r-1}}_{r-1})^{\tth_w} \to 1. \end{gather*} As the $G_{r:r}^E \subseteq G_{r:r}$ are natural linear spaces over $\ov\BF_q$ and the action of $\tth_w$ on $G_{r:r}$ is a linear automorphism, one checks directly that the restriction maps \[\BA_{E^r}^w \overset {u_E^w}  \to  (G_{r:r}^E)^{\tth_w}, \quad \BA_{E^r, w} \overset {u_{E, w}} \to G^E_{r:r} / (G_{r:r}^E)^{\tth_w}\] are isomorphisms of abelian groups. The statement then follows from the snake lemma. 
\end{proof}

\section{Deep level Deligne-Lusztig representations} \label{sec:pDL}

\subsection{Howe factorization}\label{subec:Howe}
We fix a character $\phi \colon T_r^F \to \overline\BQ_\ell^\times$ and assume that $\phi$ admits a Howe factorization $(G_i,r_i,\phi_i)_{i=-1}^d$ as in the sense of Kaletha \cite{Kal}. Namely, this means that \[\phi = \phi_{-1} \prod_{i=0}^d \phi_i |_{T_r^F},\] where $\phi_{-1}$ is a character of $T_0^F$ of depth zero, and $\L := (G_i,r_i,\phi_i)_{i=0}^d$ is {\it generic datum} defined by the following conditions:
\begin{itemize}
\item $T= G^{-1} \subseteq G^0 \subsetneq G^1 \subsetneq \cdots \subsetneq G^d = G$ are Levi subgroups of $G$ defined over $\BF_q$;
\item $0 =: r_{-1} < r_0 < \cdots < r_{d-1} \le r_d \le r$ if $d \ge 1$ and $0 \le r_0$ if $d = 0$;
\item $\phi_i: G^i(k) \to \ov\BQ_\ell^\times$ is a character of depth $r_i$, and trivial on $G_\der^i(k)$ for $-1 \le i \le d$;
\item $\phi_i$ is of depth $r_i$ and is $(G^i, G^{i+1})$-generic in the sense of \cite[\S 9]{Yu_01} for $0 \le i \le d-1$.
\end{itemize}

We put \[ L = G^0.\] Moreover, for $0 \leq i \leq d$, we set \[ s_i = \frac{r_i}{2}.\] Write $T^i_\der = G^i_\der \cap T$. We associate to the generic datum $\L$ the following subgroups of $G$.
% Let $\ov U$ be the opposite of $U$ and set $T^i_\der = G^i_\der \cap T$. We consider the following subgroups.
\begin{align*}
    \CK_{\L, r} &= G^0_r G^1_{s_0:r} \cdots G^d_{s_{d-1}:r} \\
    \CH_{\L, r} &= G^0_{0+:r} G^1_{s_0:r} \cdots G^d_{s_{d-1}:r} \\
    \CK_{\L^+, r} &= G^0_{0+:r} G^1_{s_0+:r} \cdots G^d_{s_{d-1}+:r} \\
    \CT_\L & = (T^0_\der)_{0+:r} (T^1_\der)_{r_0+:r} \cdots (T^d_\der)_{r_{d-1}+:r}\\
    \CE_{\L, r} &=[T_r\CK^+_{\L, r}, T_r\CK^+_{\L, r}]\CT_{\L, r}.
\end{align*}
Fix a Borel subgroup $B = T U$ with unipotent radical $U$ such that each $G^i$ is a standard Levi subgroup with respect to $B$. We define \[\CI_{\L, r}^\dag = \CI_{\L, B, r}^\dag = (\CK_{\L, r} \cap U_r) \CT_{\L, r} (\CK_{\L, r}^+ \cap \ov U_r),\] where $\ov U$ denotes the opposite of $U$.

\subsection{Deep level Deligne-Lusztig representations}
Let notation be as in \S\ref{subec:Howe}. The parahoric Deligne-Lusztig varieties are defined by \[X = X_{T, B, \L, r} = \{h \in G_r; h\i F(h) \in F\CI_{\L, r}^\dag\},\] which admits an action of $G_r^F \times T_r^F$ given by $(g, t): h \mapsto g h t$.

Let $H_c^i(X, \ov\BQ_\ell)$ denote the $i$th $\ell$-adic cohomology with compact support of $X$. This is a representation of $G_r^F \times T_r^F$ with respect to the above action. Let $H_c^i(X, \ov\BQ_\ell)[\phi] \subseteq H_c^i(X, \ov\BQ_\ell)$ be the subspace on which $T_r^F$ acts via the character $\phi$, which is a sub-representation of $G_r^F$. The associated (virtual) deep level Deligne-Lusztig representation of $G_r^F$ is given by \[\CR_{T_r}^{G_r}(\phi) = H_c^*(X, \ov\BQ_\ell)[\phi] := \sum_{i \in \BZ} (-1)^i H_c^i(X, \ov\BQ_\ell)[\phi].\]

Similarly, we can define \[\tX = \tX_{T, B, \L, r} = \{h \in \tG_r; h\i F(h) \in F\CI_{\L, r}^\dag\},\] where we view $\CI_{\L, r}^\dag \subseteq U_r G_{0+:r}$ as a subset of $\tG_r$ in the natural way. Note that $\tX$ admits a natural action of $\tG_r^F \times \tT_r^F$ by left/right multiplication. Let $\tphi$ be the pull-back of $\phi$ under the natural map $\tT_r^F \to T_r^F$. Then we can define a virtual $\tG_r^F$-module \[\CR_{T_r}^{\tG_r}(\tphi) = H_c^*(\tX, \ov\BQ_\ell)[\tphi]\] in a similar way.

Let $\phi^+ = \phi|_{T_{0+:r}^F}$ and let $\phi_+$ be the character of $T_r^F$ such that $\phi_+|_{T_{0+:r}^F} = \phi^+$ and $\phi_+ |_{T_0^F} = 1$. We define $\tphi^+$ and $\tphi_+$ in a similar way.
\begin{theorem} \label{Green}
    Let $g = s u \in G_r^F$ be the Jordan decomposition with $s$ semisimple and $u$ unipotent. Then \[\tr(g; \CR_{T_r}^{G_r}(\phi)) = \frac{1}{|Z^\circ(s)_r^F|} \sum_{x \in G_r^F, x\i s x \in T_r}  \phi(x\i s x) \tr(x\i u x; \CR_{T_r}^{Z^\circ(x\i s x)_r}(\phi_+)).\] Moreover, a similar equality holds with $(G_r, T_r, \phi)$ replaced by $(\tG_r, \tT_r, \tphi)$.
\end{theorem}
\begin{proof}
    We have \[\tr(g; \CR_{T_r}^{G_r}(\phi)) = \frac{1}{|T_0^F| \cdot |Z^\circ(s)_r^F|} \sum_{x \in G_r^F, x\i s x \in T_r}  \phi(x\i s x) \tr(x\i u x; H_c^*(X^{x\i s x}, \ov\BQ_\ell)[\phi^+]).\]
    In particular, if $g = u$ is unipotent, then $\tr(g; \CR_{T_r}^{G_r}(\phi))$ only depends on the restriction $\phi^+$ and hence \[\frac{1}{|T_0^F|} \tr(g; H_c^*(X, \ov\BQ_\ell)[\phi^+]) = \frac{1}{|T_0^F|} \sum_\l \tr(g; \CR_{T_r}^{G_r}(\l)) = \tr(g; \CR_{T_r}^{G_r}(\phi_+)),\] where $\l$ ranges over characters of $T_r^F$ such that $\l^+ = \phi^+$. Then the statement follows by noticing that $\CR_{T_r}^{Z^\circ(x\i s x)_r}(\phi_+) \cong H_c^*(X^{x\i s x}, \ov\BQ_\ell)[\phi_+]$.    
\end{proof}

%Let $M \subseteq G$ be the subgroup over $k$ whose affine roots are $f \in \tPhi$ such that $f(\bx) \in \BZ$. In particular, The reductive quotient of $M$ at $\bx$ is isomorphic to $G_0$. Let $\pi: M' \to M$ be the simply connected covering of the derived subgroup of $M$. Suppose $p$ is sufficiently large. Then the covering induces an isomorphism $M'_{0+:r} \cong M_{0+:r}$. Let $\tG_r = (M_{0:r}'  \ltimes G_{0+:r}) / \Delta M_{0+:r}$, where $\Delta M_{0+:r} = \{(x\i \ltimes x); x \in M_{0+:r}\}$ is a normal subgroup of $M_{0:r}'  \ltimes G_{0+:r}$. There is a group homomorphism $\tG_r \to G_r$ induced by $(m, g) \mapsto \pi(m) g$.

%Let $\tT_r \subseteq \tG_r$ be the inverse image of $T_r$ and let $\tphi$ be the pull-back of $\phi$ under the natural homomorphism $\tT_r \to T_r$. 
%\begin{lemma}The natural homomorphism $\tG_r \to G_r$ induces bijections $\tG_{0+:r} \cong G_{0+:r}$, $\tG_{r, \unip} \cong G_{r, \unip}$ and $\tG_r / \tT_{r, \red} \cong G_r / T_{r, \red}$.\end{lemma}

%By replacing the triple $(G_r, T_r, \phi)$ with $(\tG_r, \tT_r, \tphi)$ in the previous constructions, we may assume further that $G_r^\th$ is connected and hence Proposition \ref{intertwine}.

\subsection{A comparision result}
Consider the following variety \[Y = Y_{T, B, \L, r} = \{h \in G_r; h\i F(h) \in T_0 F\CI_{\L, r}^\dag\} / T_0,\] which admits an action of $G_r^F \times T_{0+:r}^F$ by left/right multiplication.

\begin{lemma} \label{equiv}
    The natural map $G_r \to G_r / T_0$ (resp. $\tG_r \to \tG_r / \tT_0 \cong G_r / T_0$) induces a $G_r^F \times T_{0+:r}^F$-equivariant isomorphism $X / T_0^F \cong Y$ (resp. $\tG_r^F \times T_{0+:r}^F$-equivariant isomorphism $\tX / \tT_0^F \cong Y$).
\end{lemma}
\begin{proof}
    We only show the second isomorphism. Let $g T_0 \in Y$. Using the natural isomorphism $\tG_r / \tT_0 \cong G_r / T_0$, we may assume further that $g = \pi(h)$ for some $h \in \tG_r$. By definition, we deduce that \[h\i F(h) \in \pi\i(g\i F(g)) \subseteq \pi\i(T_0 F \CI_{\L, r}^\dag) = \tT_0 F \CI_{\L, r}^\dag.\] By Lang's theorem there exists $t \in \tT_0$ such that $(ht)\i F(ht) \in F \CI_{\L, r}^\dag$. Hence $h t \in \tX$ and the map $\tX \to Y$ is surjective.

    Assume $h, h' \in \tX$ such that $\pi(h) T_0 = \pi(h') T_0$. Then $h' = h t$ for some $t \in \tT_0$. By definition, we have \[(ht)\i F(ht) \in (t\i F\CI_{\L, r}^\dag t t\i F(t)) \cap F\CI_{\L, r}^\dag = F\CI_{\L, r}^\dag t\i F(t) \cap  F\CI_{\L, r}^\dag,\] which implies that $t\i F(t) = 1$ since  $\tT_0 \cap F\CI_{\L, r}^\dag = \{1\}$ is trivial. Thus $t \in \tT_0^F$ and the map $\tX / \tT_0^F \to Y$ is injective as desired. 
\end{proof}

\begin{corollary} \label{mod-T}
   Suppose that $\phi$ is also trivial over $T_0^F$. Then  $\CR_{T_r}^{G_r}(\phi) \cong \CR_{\tT_r}^{\tG_r}(\tphi)$ as virtual $\tG_r^F$-modules. 
\end{corollary}
\begin{proof}
   As $\phi$ is trivial over $T_0^F$, $\tphi$ is trivial over $\tT_0^F$. Thus we have \[\CR_{T_r}^{G_r}(\phi) \cong H_c^*(X/T_0^F, \ov\BQ_\ell)[\phi^+] \cong H_c^*(\tX / \tT_0^F, \ov\BQ_\ell)[\tphi^+] \cong \CR_{\tT_r}^{\tG_r}(\tphi),\] where the second isomorphism from Lemma \ref{equiv} and the equality $\phi^+ = \phi|_{T_{0+:r}^F} = \tphi|_{T_{0+:r}^F} = \tphi^+$. 
\end{proof}

\section{The multiplicity formula} \label{sec:proof-main}
Let notation be as in \S\ref{sec:pDL}. Let $w \in \Theta_{T_r}^F$ and $t \in T_0^{\th_w} \subseteq T_r$. We define \[\tPhi_{\phi, w}(t) = \bigsqcup_{i=0}^d \{f \in \tPhi; \a_f \in \Phi_{G^i}^+ \sm \Phi_{G^{i-1}}, n_f = s_{i-1}, \a_f(t) =1, \th_w(\a_f) = -\a_f, F\i(\a_f) < 0\}.\] Put $n_w(t) = n_{G, \phi, \th, w}(t) = |\tPhi_{\phi, w}(t)|$ and $n_w = n_w(1)$. We define \[\e_w = \e_{G, \phi, \th, w}: (T_0^{\th_w})^F \to \{\pm 1\}, \quad t \mapsto (-1)^{n_w + n_w(t)}.\] Using the natural projection $T_r \to T_0$, we also view $\e_w$ as a function on $(T_r^{\th_w})^F$.
\begin{proposition}
    The function $\e_w$ is a character of $(T_r^{\th_w})^F$.
\end{proposition}
\begin{proof}
    Let $f \in \tPhi$ such that $\th_w(\a_f) = -\a_f$. Then $\a_f(T_0^{\th_w}) \in \{\pm 1\}$ and hence $\a_f(T_0^{\th_w, \circ}) = \{1\}$. The statement then follows from the proof of \cite[Proposition 2.3 (c)]{Lu90}.    
\end{proof}

For $w \in \Theta_{T_r}^F$ we define \begin{align*}
    \Theta_{T_r, \phi}^F &= \{w \in \Theta_{T_r}^F; \phi|_{(T_r^{\th_w})^F} = \e_w\}; \\ \Theta_{T_r, \phi, +}^F &= \{w \in \Theta_{T_r}^F; \phi|_{(T_{0+:r}^{\th_w})^F} = 1\}.
\end{align*}

For $w \in \Theta_{\tT_r}^F$ we put $n_w = n_{\pi(w)}$ and denote by $\te_w$ the pull-back of $\e_w$ under the natural projection $(\tT_r^{\tth_w})^F \to T_r^F$. Let $\Theta_{\tT_r, \tphi}^F = \{w \in \Theta_{\tT_r}^F; \tphi|_{(\tT_r^{\tth_w})^F} = \te_w\}$ and $\Theta_{\tT_r, \tphi, +}^F = \{w \in \Theta_{\tT_r}^F; \tphi|_{(T_{0+:r}^{\tth_w})^F} = 1\}$.

The following result is proved in \S\ref{sec:proof-tilde}.
\begin{theorem} \label{tilde-formula}
    We have \[\tag{a} \frac{1}{|(\tG_r^\tth)^F|}\sum_{g \in (\tG_r^\tth)^F} \tr(g; \CR_{\tT_r}^{\tG_r}(\tphi)) = \sum_{w \in \tT_r^F \backslash \Theta_{\tT_r, \tphi}^F / (\tG_r^\tth)^F} (-1)^{n_w}.\]
\end{theorem}

Let $\tG_{r, \unip}$ and $G_{r, \unip}$ denote the sets of unipotent elements in $\tG_r$ and $G_r$ respectively.
\begin{corollary} \label{sum-unip}
    We have \[\sum_{g \in \tG_{r, \unip}^F} \tr(g; \CR_{\tT_r}^{\tG_r}(\tphi_+)) = |\tT_r^F|\i \sum_{w \in \Theta_{\tT_r, \tphi, +}^F} |(\tT_{0+:r}^{\tth_w})^F| (-1)^{n_w}.\] 

    As a consequence, we have \[\sum_{g \in G_{r, \unip}^F} \tr(g; \CR_{T_r}^{G_r}(\phi_+)) = |T_r^F|\i \sum_{w \in \Theta_{T_r, \phi, +}^F} |(T_{0+:r}^{\th_w})^F| (-1)^{n_w},\]
\end{corollary}
\begin{proof}
    We sum the identities (a) in Theorem \ref{tilde-formula} over all characters $\tilde\l$ such that $\tilde\l^+ = \tphi^+$. By Theorem \ref{Green}, $\sum_{\tilde \l, ~ \tilde\l^+ = \tphi^+} \tr(g, \CR_{\tT_r}^{\tG_r}(\tilde\l)) = 0$ unless $g$ is unipotent. Hence the sum of left hand sides is \[\tag{i} |\tT_0^F| |(\tG_r^\th)^F|\i \sum_{g \in \tG_{r, \unip}^F} \tr(g; \CR_{\tT_r}^{\tG_r}(\tphi_+)).\]  The sum of right hands is \[\tag{ii} |\tT_r^F|\i |(\tG_r^\tth)^F|\i \sum_{w \in \Theta_{\tT_r}^F} |A(w)| |(\tT_r^{\th_w})^F| (-1)^{n_w},\] where $A(w)$ is the set of character $\tilde\l$ of $\tT_r$ such that $\tilde\l^+ = \tphi^+$ and $\tilde\l|_{(\tT_r^{\tth_w})^F} = \te_w$. The first statement then follows from the equality between (i) and (ii) by noticing that $|A(w)| = 0$ if $\tphi |_{(\tT_{0+:r}^{\th_{\tw}})^F} \neq 1$, and $|A(w)| = |(\tT_0^{\tth_{\tw}})^F|\i |\tT_0^F|$ otherwise.  

    Now we show that second statement. Notice that $\tT_{0+:r} \cong T_{0+:r}$, $\tphi_+ = \phi_+$ and $\tG_{r, \unip} \cong G_{r, \unip}$. Applying Corollary \ref{mod-T} we have \[\sum_{g \in \tG_{r, \unip}^F} \tr(g; \CR_{\tT_r}^{\tG_r}(\tphi_+)) = \sum_{g \in G_{r, \unip}^F} \tr(g; \CR_{T_r}^{G_r}(\phi_+)).\] By Lemma \ref{Theta}, we have $\tT_r^F \backslash \Theta_{\tT_r}^F \cong T_r^F \backslash \Theta_{T_r}^F$ and hence $\tT_r^F \backslash \Theta_{\tT_r, \tphi, +}^F \cong T_r^F \backslash \Theta_{T_r, \phi, +}^F$. Moreover, as $(\tT_{0+:r}^{\tth_w})^F \cong (T_{0+:r}^{\th_w})^F$ and $n_w = n_{\pi(w)}$ for $w \in \Theta_{\tT_r}^F$, we have \[|\tT_r^F|\i \sum_{w \in \Theta_{\tT_r, \tphi, +}^F} |(\tT_{0+:r}^{\tth_w})^F| (-1)^{n_w} = |T_r^F|\i \sum_{w \in \Theta_{T_r, \phi, +}^F} |(\tT_{0+:r}^{\th_w})^F| (-1)^{n_w}.\] Here the statement follows.
\end{proof}

Now we prove the main result of the paper.
\begin{theorem} \label{formula}
    We have \[\frac{1}{|(G_r^\th)^F|}\sum_{g \in (G_r^\th)^F} \tr(g; \CR_{T_r}^{G_r}(\phi)) = \sum_{w \in T_r^F \backslash \Theta_{T_r, \phi}^F / (G_r^\th)^F} (-1)^{n_w}.\] 
\end{theorem}
\begin{proof}
    We have \begin{align*}
        &\quad\ |(G_r^\th)^F|\i \sum_{g \in (G_r^\th)^F} \tr(g; \CR_{T_r}^{G_r}(\phi)) \\ &= |(G_r^\th)^F|\i \sum_{s \in (G_{r, ss.}^\th)^F} \sum_{x \in G_r^F, x\i s x \in T_r} |Z^\circ(s)_r^F|\i \phi(x\i s x) \\ &\quad\ \times \sum_{u \in (Z^\circ(s)_{r, \unip}^\th)^F} \tr(x\i u x; \CR_{T_r}^{Z^\circ(x\i s x)_r}(\phi_+)) \\ &=|(G_r^\th)^F|\i \sum_{s \in (G_{r, ss.}^\th)^F} \sum_{x \in G_r^F, x\i s x \in T_r} |Z^\circ(s)_r^F|\i \phi(x\i s x) \\ &\quad\ \times |T_r^F|\i \sum_{y \in Z^\circ(x\i s x)_r^F, y \in \Theta_{T_r, \th_{x\i}, \phi, +}^F} |(T_{0+:r}^{\th_{x\i}})^F| (-1)^{n_{Z^\circ(x\i s x), \phi, \th_{x\i}, y}} \\ &=|(G_r^\th)^F|\i  \sum_{t \in T_0^F} \phi(t) |Z^\circ(t)_r^F|\i \sum_{v \in G_r^F, v\i t v \in G_r^\th} |T_r^F|\i \sum_{y \in Z^\circ(t)^F, yv \in \Theta_{T_r, \phi, +}^F} |(T_{0+:r}^{\th_{yv}})^F| (-1)^{n_{yv}(t)} \\ &=|(G_r^\th)^F|\i |T_r^F|\i \sum_{t \in T_0^F} \phi(t) \sum_{w \in \Theta_{T_{0+:r}, \phi}^F, t \in G_r^{\th_w}} |(T_{0+:r}^{\th_w})^F| (-1)^{n_w(t)} \\ &=|(G_r^\th)^F|\i |T_r^F|\i \sum_{w \in \Theta_{T_r, \phi, +}^F} (-1)^{n_w} |(T_{0+:r}^{\th_w})^F| \sum_{t \in (T_0^{\th_w})^F} \phi(t) \e_w(t) \\ &=|(G_r^\th)^F|\i |T_r^F|\i \sum_{w \in \Theta_{T_r, \phi}^F} |(T_r^{\th_w})^F| (-1)^{n_w} \\ &= \sum_{w \in T_r^F \backslash \Theta_{T_r, \phi}^F / (G_r^\th)^F}  (-1)^{n_w},
    \end{align*} where the second equality follows from Lemma \ref{stable} and Corollary \ref{sum-unip} by taking $(G, \th, T,\phi) = (Z^\circ(t), \th_{x\i}, T, \phi)$; the third one uses the change of variables $(v, t) = (x\i, v s v\i)$ and the equality $n_{Z^0(t), \phi, \th_{x\i}, y}(t) = n_{G, \phi, \th, yv}(t) = n_{yv}(t)$; the fourth one uses the change of variable $w = y v$ and that $(y\i w)\i t (y\i w) = w\i t w \in G_r^\th$ if and only if $t \in G_r^{\th_w}$ for any $y \in Z^\circ(t)_r^F$; the fifth one follows from the definition $\e_w(t) = (-1)^{n_w + n_w(t)}$ for $t \in (T_0^{\th_w})^F$.    
\end{proof}

\section{A vanishing result} \label{sec:vanish}
Consider the variety \[\tSigma = \{(x, \g) \in F \CI_{\L, r}^\dag \times \tG_r/\tG_r^\tth; F(\g) = x \g\},\] on which $\tT_r^F$ acts by $t: (x, \g) \mapsto (txt\i, t\g)$.

\begin{proposition} \label{intertwine}
    The map $h \mapsto (F(h)\i h, h\i \tG_r^\tth)$ gives an isomorphism \[\tG_r^F \backslash \tX \cong \tSigma.\] In particular, $\dim \Hom_{(\tG_r^\th)^F}(R_{\tT_r}^{\tG_r}(\tphi), 1) = \dim H_c^*(\tSigma, \ov\BQ_\ell)[\tphi]$.
\end{proposition}
\begin{proof}
    It is clear that the map is injective. We show it is surjective. Let $(x, \g) \in \tSigma$ with $\g = y \tG_r^\tth$ for some $y \in \tG_r$. Then $F(y)\i x y \in \tG_r^\tth$. As $\tG_r^\tth$ is connected and $F$-stable, there exists $z \in \tG_r^\tth$ such that $F(y\i) x y = F(z) z\i$. Let $h = (y z)\i$. Then $F(h)\i h = x \in F\CI_{\L, r}^\dag$, $h \in \tX$ and $\g = h\i \tG_r^\tth$ as desired.    
\end{proof}

By Lemma \ref{decomposition}, there is a decomposition \[\tSigma = \sqcup_{w \in \tT_r \backslash \Theta_{\tT_r} / \tG_r^\th} \tSigma_w,\] where $\tSigma_w$ consists of elements $(x, \g) \in \tSigma$ such that $\g \in \tCC_w$.

Fix $w \in \Theta_{T_r}$. We write $A_w = A_{E, w}$, $\BA_w= \BA_{E, w}$, $u_w = u_{E, w}$ and so on for $E = \tPhi_r^{>0}$, see \S\ref{subsec:para}.
\begin{proposition} \label{isomorphism}
    For $w \in \Theta_{\tT_r}$ the map $(x, \psi) \to u_w(x) \psi$ gives an isomorphism $\BA_w \times \tCO_w \cong \tCC_w$.
\end{proposition}
\begin{proof}
    By Lemma \ref{parameter} the map is surjective. It remains to show it is injective. Let $x, x' \in \BA_w$ and $s, s' \in \tT_0$ such that $u_w(x) s w \tG_r^\tth = u_w(x') s' w \tG_r^\tth$. Then we have $U_0 \tilde p_0(s) \tG_0^{\tth_w} = U_0 \tilde p_0(s')  \tG_0^{\tth_w}$. This implies that $\tilde p_0(s\i s') \in \tG_0^{\tth_w}$, that is, $s\i s' \tth_w(s\i s')\i \in G_{0:+}$. As $s\i s' \tth_w(s\i s')\i \in \tT_0$, we have $s\i s' = \tth_w(s\i s')$, and hence we may assume that $s = s'$. Then the injectivity also follows from Lemma \ref{parameter}. 
\end{proof}

Using Proposition \ref{isomorphism} we may write the subset $\tSigma_w$ as \[\tSigma_w = \{(x, v\psi) \in \CI_{\L, r}^\dag \times u_w(\BA_w)\tCO_w;  F(x v\psi) = v\psi\},\] on which $\tT_r^F$ acts by $t: (x, v\psi) \mapsto (txt\i, t v\psi)$. We set \[\CI_{\L, r}^+ = \CI_{\L, r}^\dag T_{0+:r} = T_{0+:r} \CI_{\L, r}^\dag \text{ and } \tPhi_{\L, r} = \{f \in \tPhi_r^{>0}; G_r^f \subseteq \CI_{\L, r}^+\}.\] For $0 \le i \le d$, let $\tPhi_{w, i} = (A_w \sm C_w) \cap (\tPhi_{G^i} \cup \tPhi_{\L, r})$, and let $\tSigma_w^i$ be the set of pairs $(x, v\psi) \in \tSigma_w$ such that $v_f = 0$ unless $f \in \tPhi_{w, i}$. Here for $v \in u_w(\BA_w)$ we denote by $v_f \in \BA_{f, w}$ the coordinate of $u_w\i(v) \in \BA_w$ at $f \in A_w \sm C_w$.

The following result is proved at the end of this section.
\begin{proposition} \label{cut}
    Let $1 \le i \le d$. We have $H_c^*(\tSigma_w^i, \ov\BQ_\ell)[\tphi] = 0$ unless $\tth_w(\Phi_{G^{i-1}}) = \Phi_{G^{i-1}}$, in which case $H_c^*(\tSigma_w^i, \ov\BQ_\ell)[\tphi] \cong H_c^*(\tSigma_w^{i-1}, \ov\BQ_\ell)[\tphi]$. 
\end{proposition}

As a consequence, we have the following main result of this section. 
\begin{corollary} \label{red-to-CI}
    We have $H_c^*(\tSigma_w, \ov\BQ_\ell)[\tphi] = 0$ unless $\tth_w(\Phi_{G^i}) = \Phi_{G^i}$ for $0 \le i \le d$, in which case $H_c^*(\tSigma_w, \ov\BQ_\ell)[\tphi] \cong H_c^*(\tSigma_w^0, \ov\BQ_\ell)[\tphi]$.
\end{corollary}

Let $E_{w, i} = \tPhi_{w, i} \sm \tPhi_{w, i-1}$, which by definition consists of affine roots $f \in (\tPhi_{G^i} \cap A_w) \sm (\tPhi_{G^{i-1}} \cup C_w)$ such that either $n_f < s_{i-1} = r_{i-1}/2$ or $n_f = s_{i-1}$ and $\a_f < 0$. Then \[\tSigma_w^i \sm \tSigma_w^{i-1} = \bigsqcup_{f \in E_{w, i}} \tSigma_w^{i, f},\] where $\tSigma_w^{i, f}$ consists of $(x, v\psi) \in \tSigma_w^i$ such that $v_f \neq 0$ and $v_{f'} = 0$ for all $f' < f \in E_{w, i}$. 

\begin{lemma} \label{sum}
    Let $f, f_1, f_2 \in \tPhi_{G^i}$ such that $f \ge f_1 + f_2$ and $f_1, f_2 \in \tPhi_{\L, r}$. Then $f \in \tPhi_{\L, r}$ or $n_f > r$. 
\end{lemma}

For $\psi \in \tCO_w$ we set $\tth_\psi = \tth_x$ for some/any $x \in \tG_r$ such that $\psi = x \tG_r^\th$.  
\begin{lemma} \label{conjugation}
    Let $f \in E_{w, i}$ such that $f \neq \tth_w(f) \in \tPhi_{G^i} \sm \tPhi_{G^{i-1}}$ and let $v \in u_w(\BA_w) \cap \tCI_{\L, r}^i$ such that $v_f \neq 0$ and $v_{f'} = 0$ for all $f' < f \in E_{w, i}$. Consider the map \begin{align*} \xi_\psi: \BA^1 &\to G_r^{\tth_\psi} \\ z &\mapsto u_{r_{i-1}-f}(z) {}^{\tth_\psi}u_{r_{i-1}-f}(z) [{}^{\tth_\psi}u_{r_{i-1}-f}(z)\i, u_{r_{i-1}-f}(z)\i]^{\frac{1}{2}}.\end{align*} Then we have \[v \xi_\psi(z) v\i \in \CI_{\L, r}^\dag \a_f^\vee(1 + z v_f \varpi^{r_{i-1}}).\]
\end{lemma}
\begin{proof}
    Let $f_1 = r_{i-1} - f,  f_2 = r_{i-1} - \tth_w(f) \in \Phi_{G^i} \sm \Phi_{G^{i-1}}$. We claim that 
    
    \[\tag{a} f_1, f_2 \in \tPhi_{\L, r}.\]
    
    Indeed, if $n_f < s_{i-1}$, then $n_{f_i} = r_{i-1} - n_f > s_{i-1}$ and the claim follows. Otherwise, we have $n_f = s_{i-1}$ and $\a_f < 0$ by the definition of $E_{w, i}$. Moreover, as $f \in A_w$, we have $\tth_w(f) < f$ and hence $\a_{\tth_w(f)} < 0$. So the claim also holds.  

    Let $E_{w, i}^f = \{g \in E_{w, i}; g \ge f\}$. By the commutator relation of affine root subgroups we have \[v \xi_\psi(z) v\i = \a_f^\vee(1 + z v_f \varpi^{r_{i-1}}) \d,\] where $\d$ is product of elements $x_h$ in the affine root subgroups $G_r^{h}$ such that \[h = c_1 f_1 + c_2 f_2 + \sum_{g \in E_{w, i}^f \sqcup  \tPhi_{w, i-1}} c_g g\] such that $c_1, c_2, c_f \in \BZ_{\ge 0}$, $c_1 + c_2 \ge 1$ and the case that $c_1 = c_f = 1$, $c_2 = c_g = 0$ for $f \neq g \in E_{w, i}^f \sqcup  \tPhi_{w, i-1}$ is excluded. 
    
    It suffices to show $x_h \in \CI_{\L, r}^\dag$. If $c_g = 0$ for $g \in E_{w, i}^f \sqcup  \tPhi_{w, i-1}$, the statement follows from the claim (a). Otherwise, let $i \le i_0 \le d$ be the maximal integer such that $c_{g_0} \ge 1$ for some $g_0 \in \tPhi_{G^{i_0}} \cap (E_{w, i}^f \sqcup  \tPhi_{w, i-1})$. In particular, $h \in \tPhi_{G^{i_0}}$. First assume that $i < i_0$. By the definition of $\tPhi_{w, i}$, we have $g_0 \in \tPhi_{\L, r}$ and hence $n_{g_0} \ge s_{i_0-1}$. If $h \in \Phi \times \BZ$, then $n_h \ge n_{g_0} + n_{f_1} > s_{i_0-1}$ and hence $h \in \tPhi_{\L, r}$ as desired. If $h \in \BZ_{\ge 1}$, then $x_h \in \b^\vee(1 + \varpi^h \ov\BF_q[[\varpi]])$ for some $\b \in \Phi_{G^{i_0}}$. In this case, there exists $g_0 \neq g_0' \in \tPhi_{G^{i_0}}$ such that $c_{g_0'} \ge 1$. Thus $h \ge n_{g_0} + n_{g_0'} + n_{f_1} > r_{i_0-1}$ and hence $x_h \in \CI_{\L, r}^\dag$ as desired. Now we assume $i_0 = i$ and hence $h \in \tPhi_{G^i}$. If $h \in \Phi \times \BZ$, then we have $n_h \ge n_{f_1} \ge s_{i-1}$ and either $h \ge f_1$ or $h \ge f_2$. This implies that $h \subseteq \tPhi_\L$ as desired. If $h \in \BZ_{\ge 1}$, then $x_h \in \a^\vee(1 + \varpi^h \ov\BF_q[[\varpi]])$ for some $\a \in \Phi_{G^{i_0}}$. We show that $x_h \in \CI_{\L, r}^\dag$ unless $c_1 = c_f = 1$ and $c_2 = c_g = 0$ for $f \neq g \in E_{w, i}^f \sqcup  \tPhi_{w, i-1}$. Indeed, as $c_1 + c_2 \ge 1$ and $f_1, f_2 \in \tPhi_{G^i} \sm \tPhi_{G^{i-1}}$ we have either $c_1 + c_2 \ge 2$ or $c_{g'} \ge 1$ for some $g' \in (E_{w, i}^f \sqcup  \tPhi_{w, i-1}) \cap (\tPhi_{G^i} \sm \tPhi_{G^{i-1}})$. Assume $x_h \notin \CI_{\L, r}^\dag$. Then we have $c_1 + c_2 = 1$ and $g' \in E_{w, i}^f$ by Lemma \ref{sum}. In particular, $h \ge n_{f_1} + n_{g'} \ge n_{f_1} + n_{f} = r_{i-1} - n_f + n_f = r_{i-1}$. As $x_h \notin \CI_{\L, r}^\dag$, it follows that $h = r_{i-1}$.  If $c_2 = 1$, then $h \ge f_2 + g'$ and hence $\tth_w(f) \ge g' \ge f$, contradicting the assumption that $f > \tth_w(f)$. Otherwise, $c_1 = 1$ and $h \ge f_1 + g'$. Thus we have $f \ge g' \ge f$, $g' = f$ and $h = f_1 + g'$, that is, $c_1 = c_f = 1$ and $c_2 = c_g = 0$ for $f \neq g \in E_{w, i}^f \sqcup  \tPhi_{w, i-1}$. The proof is finished.    
\end{proof}

Fix $N \in \BZ_{\ge 1}$ such that $F^N$ acts trivially on $X_*(T)$. For $m \in \BZ_{\ge 1}$ and $\l \in X_*(T)$ consider the following subgroups \begin{align*}
    H_{\l, m} &= \{t \in T(1 + \varpi^m \ov\BF_q[[\varpi]]); L(t) \in \l(1 + \varpi^m \ov\BF_q[[\varpi]])\} \subseteq T_{m:r}; \\ M_{\l, m} &= \{x F(x) \cdots F^{N-1}(x); x \in \l(1 + \varpi^m \ov\BF_{q^N}[[\varpi]])\}  \subseteq H_{\l, m}^\circ \cap T_r^F.
\end{align*} Here $L(t) = t\i F(t)$ for $t \in \tT_r$.

\begin{lemma}
    For $1 \le i \le d$ and $\a \in \Phi_{G^i} \sm \Phi_{G^{i-1}}$ the character $\tphi$ is non-trivial over $M_{\a, r_{i-1}}$.
\end{lemma}
\begin{proof}
    Note that $\tphi^+ = \phi^+$. Then the statement follows from the construction of How factorization of $\phi$ in \cite{Kal}.
\end{proof}

For $f \in E_{w, i}$ we denote by \[p_f: \a_f^\vee(1 + \varpi^{r_{i-1}}\ov\BF_q[[\varpi]]) \to \a_f^\vee(1 + \varpi^{r_{i-1}} \ov\BF_q[[\varpi]]) / \a_f^\vee(1 + \varpi^{r_{i-1}+1} \ov\BF_q[[\varpi]]) \cong \BA^1\] the natural projection. By definition we have $\a_f^\vee(1 + p_f(\t) \varpi^{r_{i-1}}) \t\i \in \a_f^\vee(1 + \varpi^{r_{i-1}+1} \ov\BF_q[[\varpi]]) \subseteq \CI_{\L, r}^\dag$.

\begin{lemma} \label{extension}
    Let $f \in E_{w, i}$ such that $f, \tth_w(f) \in \tPhi_{G^i} \sm \tPhi_{G^{i-1}}$. Then there exists a morphism $\D: H_{\a_f, r_{i-1}} \times \tSigma_{w, f}^i \to \tSigma_{w, f}^i$ such that for each $t \in H_{\a_f, r_{i-1}}$ the restriction map $\D_t = \D(t, -)$ is an automorphism of $\tSigma_w^i$, which coincides with the action of $T_r^F$ on $\tSigma_w^i$ when $t \in H_{\a_f, r_{i-1}}^F$.
\end{lemma}
\begin{proof}
    Case (1):  $f \neq \tth_w(f)$. Let $\xi$ be the map  defined in Lemma \ref{conjugation}. We define \[\D: H_{\a_f, r_{i-1}} \times \tSigma_{w, f}^i \to \tSigma_{w, f}^i, \quad (t, x, v\psi) \mapsto (x_{t, v\psi}, F(t) v\psi),\] where \[x_{t, v\psi} = {}^t (x v \xi_\psi(v_f\i p_f(L(t))) v\i L(t)\i).\] By Lemma \ref{conjugation} and the definition of $p_f$, we have \[v \xi_\psi(v_f\i p_f(L(t))) v\i L(t)\i \in \CI_{\L, r}^\dag \a_f^\vee(1 + p_f(L(t)) \varpi^{r_{i-1}}) L(t)\i \subseteq \CI_{\L, r}^\dag.\] Hence $x_{t, v\psi} \in \CI_{\L, r}^\dag$. Moreover, if $t \in H_{\a_f, r_{i-1}}^F$ we have $L(t) = 1$, $x_{t, v\psi}  = t x t\i$ and hence the endomorphism $\D_t$ of $\tSigma_{w, f}^i$ coincides with action of $\tT_r^F$.

    Now we show $\D$ is well-defined, that is, \[\tag{a} F(x_{t, v\psi} F(t) v\psi) = F(t) v\psi.\] By definition, we have $F(x v\psi) = v \psi$. Hence (a) is equivalent to \[\tag{b} x_{t, v\psi} F(t) v\psi = t x v\psi,\] which follows by inserting the expression of $x_{t, v\psi}$ and using that $\Im\xi \subseteq \tG_r^{\tth_\psi}$. 

    Finally we show $\D_t$ is an isomorphism of $\tSigma_{w, f}^i$ for $t \in H_{\a_f, r_{i-1}}$. By definition the composition $\D_{t} \circ \D_{t\i}$ sends $(x, v\psi)$ to $(x \d_{t, v\psi}, v\psi)$ for some element $\d_{t, v\psi} \in \tG_r$. In particular, $F(x \d_{t, v\psi} v\psi) = v\psi = F(x v\psi)$, that is, $v\i \d_{t, v\psi} v \in \tG_r^{\tth_\psi}$. Thus $\D_{t} \circ \D_{t\i}$ has an inverse given by $(x, v\psi) \mapsto (x \d_{t, v\psi}\i, v\psi)$. Hence $\D_t$ (resp. $\D_{t\i}$) has a right (resp. left) inverse. By symmetry, $\D_t$ also has a left inverse and hence is invertible as desired.

    Case (2): $f = \tth_w(f)$. In particular, $L(t) \in \a_f^\vee(1 + \varpi^{r_{i-1}} \ov\BF_q[[\varpi]]) \subseteq \tG_r^{\tth_\psi}$ for any $t \in H_{\a_f, r_{i-1}}$ and $\psi \in \tCO_w$. We define \[\D: H_{\a_f, r_{i-1}} \times \tSigma_{w, f}^i \to \tSigma_{w, f}^i, \quad (t, x, v\psi) \mapsto (x_{t, v\psi}, F(t)v\psi),\] where \[x_{t, v\psi} = {}^t (x [v, L(t)]).\] Using the inclusions $[v, L(t)] \in \CI_{\L, r}^\dag$ and $L(t) \in \tG_r^{\tth_\psi}$, we deduce that $\D$ is well-defined. The remaining statements follow as in Case (1) in a similar way.     
\end{proof}

\begin{proof}[Proof of Proposition \ref{cut}]
    Suppose that there exists $\b \in \Phi_{G^i} \sm \Phi_{G^{i-1}}$ such that $\a := \tth_w(\b) \in \Phi_{G^{i-1}}$. As in Case (2) of Lemma \ref{extension}, there is a morphism \[\D: H_{\a^\vee +\b^\vee, r_{i-1}} \times \tSigma_w^i \to \tSigma_w^i, \quad (t, x, v\psi) \mapsto (x_{t, v\psi}, F(t) v\psi),\] where \[x_{t, v\psi} = {}^t (x [v, L(t)]).\] In particular, for each $t \in M_{\a^\vee +\b^\vee, r_{i-1}} \subseteq H_{\a^\vee +\b^\vee, r_{i-1}}^\circ$ the induced endomorphism on the cohomology groups $H_c^i(\tSigma_w^i, \ov\BQ_\ell)$ is trivial. As $\a \in \Phi_{G^{i-1}}$, $\tphi$ is trivial over $M_{\a, r_{i-1}}$ and hence $\tphi(M_{\a^\vee +\b^\vee, r_{i-1}}) = \tphi(M_{\b^\vee, r_{i-1}}) \neq \{1\}$. Thus we have $H_c^*(\tSigma_w^i, \ov\BQ_\ell)[\tphi] =0$ as desired.

    Now we assume that $\tth_w(\Phi_{G^{i-1}}) = \Phi_{G^{i-1}}$. As $\tSigma_w^i \sm \tSigma_w^{i-1} = \bigsqcup_{f \in E_{w, i}} \tSigma_{w, f}^i$  It suffices to show $H_c^*(\tSigma_{w, f}^i, \ov\BQ_\ell)[\tphi] = 0$ for all $f \in E_{w, i}$. Let $\D: H_{\a_f, r_{i-1}} \times \tSigma_{w, f}^i \to \tSigma_{w, f}^i$ be the morphism defined in Lemma \ref{extension}. Then for each $t \in H_{\a_f, r_{i-1}}^\circ$ the induced endomorphism on the cohomology groups $H_c^i(\tSigma_{w, f}^i, \ov\BQ_\ell)$ is trivial. As $\tphi$ is nontrivial on $M_{\a_f, r_{i-1}} \subseteq H_{\a_f, r_{i-1}}^\circ$, we have $H_c^*(\tSigma_{w, f}^i, \ov\BQ_\ell)[\tphi] = 0$ as desired.
\end{proof}

\section{Reduction to the study of $\Sigma_w^\tBK$} \label{sec:first-reduction}
Thanks to Corollary \ref{red-to-CI}, it remains to compute the cohomology of $\tSigma_w^0$. To this end, we reduce it to the computation for a much simpler variety. We fix $w \in \Theta_{T_r}$ such that $\tth_w(\Phi_{G^i}) = \Phi_{G^i}$ for $0 \le i \le d$. 

\subsection{The variety $\tSigma_w'$}
Consider the variety \[\tSigma_w' = \{(x, y, \t, \psi) \in \CI_{\L, r}^\dag \times \CI_{\L, r}^\dag \times T_{0+:r} \times \tCO_w; F(x \t \psi) = y \t \psi\},\] on which $\tT_r^F = \tT_0^F \times T_{0+:r}^F$ acts by $t = t_s t_u: (x, y, \t, \psi) = (t x t\i, t y t\i, t_u \t, t_s \psi)$.

Recall that $\CI_{\L, r}^+ = \CI_{\L, r}^\dag T_{0+:r}$ and $\tPhi_{\L, r} = \{f \in \tPhi_r^{>0}; G_r^f \subseteq \CI_{\L, r}^+\}$. 
\begin{lemma} \label{admissible}
    We have that $(\CI_{\L, r}^+)^{\tth_\psi}$ are affine spaces of the the same dimension for $\psi \in \tCO_w$.
\end{lemma}
\begin{proof}
    By assumption, $\tth_w(\Phi_{G^i}) = \Phi_{G^i}$ for $0 \le i \le d$, which implies that $\tPhi_{\L, r}$ is $\tth_w$-admissible. Then the  statement follows from Lemma \ref{parameter} by noticing that $\CI_{\L, r}^+ = G_r^{\tPhi_{\L, r}}$.
\end{proof}

\begin{lemma} \label{red-to-CK}
    Let $\tSigma_w^0$ be as in Corollary \ref{red-to-CI}. The map $q: \tSigma_w' \to \tSigma_w^0$ given by $(x, y, \t, \psi) \mapsto (xy\i, y\t\psi)$ is an $\tT_r^F$-equivariant affine space fibration. 
\end{lemma}
\begin{proof}
    By Lemma \ref{admissible} and Lemma \ref{parameter}, the map $q$ is well-defined. Let $(z, v\psi) \in \tSigma_w^0$. In particular, $z \in \CI_{\L, r}^\dag$ and $v \in \CI_{\L, r}^+$. Suppose $(x, y, \t, \psi') \in q\i(x, v\psi)$. Then $\psi' = \psi$ by Proposition \ref{isomorphism}. Hence $y \t \in v (\CI_{\L, r}^+)^{\tth_\psi}$ and $x = zy$. Thus $q\i(x, v\psi)$ is isomorphic to the inverse image of $v (\CI_{\L, r}^+)^{\tth_\psi}$ under the natural product map $m: \CI_{\L, r}^\dag \times T_{0+:r} \to \CI_{\L, r}^+$. As $m$ is a trivial affine space fibration with fibers isomorphic to $\CT_{\L, r}$, it follows from Lemma \ref{admissible} that $q\i(x, v\psi)$ is also an affine space.  
\end{proof}

\begin{lemma} \label{F-fixe}
    We have $H_c^*(\tSigma_w', \ov\BQ_\ell)[\tphi] = 0$ unless $F \tCO_w = T_{0+:r} \tCO_w$.
\end{lemma}
\begin{proof}
    Consider the subgroup $H = \{t \in \tT_0; F(t\i) t \in \tG_r^{\tth_w}\}$ and its action on $\tSigma_w'$ given by $t: (x, y, \t, \psi) \mapsto (t x t\i, F(t) y F(t)\i, \t, t\psi)$. As this action of $H$ commutes with the action of $T_r^F$, we have $H_c^*(\tSigma_w', \ov\BQ_\ell)[\tphi] = H_c^*((\tSigma_w')^{H^\circ}, \ov\BQ_\ell)[\tphi]$. Thus $H_c^*(\tSigma_w', \ov\BQ_\ell)[\tphi] = 0$ unless $(\tSigma_w')^{H^\circ} \neq \emptyset$. In this case, the $H^\circ$-fixed point set of the image of $\tSigma_w'$ under the natural projection map $\tG_r \to \tG_0$ is also non-empty. By \cite[Proposition 8.3]{Lu90}, we have $F(w) \in G_{0+:r} t w \tG_r^{\tth}$ for some $t \in \tT_r$. Applying Lemma \ref{Theta} we have $F(w) \in T_{0+:r} t w \tG_r^{\tth}$ as desired.    
\end{proof}

\subsection{The variety $\Sigma_w^\tBK$} \label{subsec:BK}
Now we assume further that $F(w) \in w \tG_r^\tth$. In particular, $F \tth_w = \tth_w F$. Let  $\tilde\CK_{\L, r} = \pi\i(\CK_{\L, r})$. Notice that $\tth_w$ preserves $\tilde\CK_{\L, r}$, $\CH_{\L, r}$ and $\CE_{\L, r}$. 

Let $\tBK = \tilde\CK_{\L, r} / \CE_{\L, r}$. We denote by $\tBL$, $\tBT$, $\BT^+$, $\BI$ and $\BH$ the natural images of $\tL_r = \pi\i(L_r)$, $\tT_r$, $T_{0+:r}$, $\CI_{\L, r}^\dag$ and $\CH_{\L, r}$ in $\tBK$ respectively. Then $\tBK = \tBL \BH$ and \[\BH =  \BT^+ \prod_{\a \in \Psi} \BH_\a,\] where $\Psi$ is the set of roots $\a \in \Phi_{G^i} \sm \Phi_{G^{i-1}}$ with $1 \le i \le d$ such that $s_{i-1} \in \BZ$ and $\BH_\a \cong \BA^1$ is the image of $G_r^{\a + s_{i-1}}$ in $\BK$. Moreover, we have that $[\BT^+, \BH]=0$, $[\BH_\a, \BH_\b] = 0$ if $\a + \b \neq 0$ and $[\BH_\a, \BH_\b] = \a^\vee(1 + \varpi^{r_i} \ov\BF_q) \subseteq \BT^+$ if $\a + \b = 0$ and $\BH_\a \cong \BA^1$.

Let $ \tL_0 = \tBK / \BH$, which is isomorphic to the Levi subgroup $\pi\i(L) \subseteq G_\sc$.
\begin{lemma} \label{injection}
The natural projections \[\tCO_w \cong \tT_0 \tilde\CK_{\L, r}^{\tth_w} / \tilde\CK_{\L, r}^{\tth_w} \to \tT_0 \tBK^{\tth_w} / \tBK^{\tth_w} \to \tT_0 \tL_0^{\tth_w} / \tL_0^{\tth_w} \cong \tT_0 / (\tT_0)^{\tth_w}\] are isomorphisms.\end{lemma}
\begin{proof} 
It suffices to show the composition is injective. Let $s_1, s_2 \in \tT_0$ such that $s_1 \tL_0^{\tth_w} = s_2 \tL_0^{\tth_w} \in \tL_0 / (\tL_0)^{\tth_w}$. Let $s = s_1\i s_2 \in \tT_0$. By definition, $s\i \tth_w(s) \in \CH_{\L, r}$ is unipotent in $\tilde \CK_{\L, r}$. On the other hand, $s\i \tth_w(s)$ is a semisimple element of $\tilde\CK_{\L, r}$. Thus $s\i \tth_w(s) = 1$ and $s \in \tilde\CK_{\L, r}^{\th_w}$, which means that $s_1 \tilde\CK_{\L, r}^{\th_w} = s_2 \tilde\CK_{\L, r}^{\th_w}$ as desired.\end{proof}

Let $\BL_U$ and $\BH_U$ be the images of $L_r \cap U_r$ and $\CH_{\L, r} \cap U_r$ in $\tBK$ respectively. Note that $\BI = \BL_U \BH_U$. Consider the variety \[\Sigma_w^{\prime, \tBK}=\{(x, y, \t, \psi) \in \BL_U\BH_U \times \BL_U\BH_U \times \BT^+ \times \tCO_w; F(x \t \psi) = y \t \psi\}.\] Here we use the natural identification $\tT_0 \tBK^{\tth_w}/\tBK^{\tth_w} \cong \tCO_w$ in Lemma \ref{injection}.

\begin{lemma} \label{red-to-BK}
    The map $\tSigma_w' \to \Sigma_w^{\prime, \tBK}$ induced by the natural projection $\tilde\CK_{\L, r} \to \tBK$ is a $\tT_r^F$-equivariant affine space fibration. 
\end{lemma}
\begin{proof}
    By Lemma \ref{injection} one computes that the fibers are isomorphic to $\CT_{\L, r} \times \CE_{\L, r} \times \CE_{\L, r}^{\tth_\psi}$ for $\psi \in \tCO_w$. As $\CE_{\L, r}$ is preserved by $\tth_w$, by a similar argument in Lemma \ref{parameter}, $\CE_{\L, r}^{\tth_\psi}$ are affine paces of the same dimension. So the statement follows.
\end{proof}

Let $(x, y, \t, \psi) \in \Sigma_w^{\prime, \tBK}$. Write $\t\i F(x) \t = x_1\i x_0\i$ and $\t\i y \t = y_0 y_1$ uniquely, where $(x_0, y_0, x_1, y_2) \in F\BL_U \times \BL_U \times F\BH_U \times \BH_U$. Then the variety $\Sigma_w^{\prime, \tBK}$ is defined by \[F(\psi) = x_0 x_1 L(\t) y_0 y_1 \psi.\] Here $L(\t) = F(\t)\i \t$. As $\BT^+$ lies in the center of $\tBK$, by replacing $x_1$ with $y_0\i x_1 y_0 y_1$, the above equality is equivalent to that \[F(\psi) = x_0 y_0 L(\t) x_1 y_1 \psi.\] Let $\Sigma_w^\BK$ be the set of tuples \[\{(x_0, y_0, z, \t, \psi) \in F\BL_U \times \BL_U \times F\BH_U\BH_U \times \BT^+ \times \tCO_w\}\] such that \[F(\psi) = x_0 y_0 L(\t) z \psi.\] Then $\tT_r^F$ acts on $\Sigma_w^\tBK$ by $t=t_st_u: (x_0, y_0, z, \t, \psi) \mapsto (t_s x_0 t_s\i,  t_s y_0 t_s\i, z, t_u \t, t_s\psi)$. Then the map $(x_1, y_1) \to z = z_1 y_1$ induces a natural $\tT_r^F$-equivariant affine space fibration $\Sigma_w^{\prime, \tBK} \to \Sigma_w^\tBK$.

\begin{lemma} \label{trivial}
    We have $H_c^*(\tSigma_w, \ov\BQ_\ell)[\tphi] = 0$ unless $w \in \Theta_{\tT, \tphi, +}^F$ and $\tth_w(\Phi_{G^i}) = \Phi_{G^i}$ for $0 \le i \le d$, in which case $H_c^*(\tSigma_w, \ov\BQ_\ell)[\tphi] = H_c^*(\Sigma_w^\tBK, \ov\BQ_\ell)[\tphi]$. 
\end{lemma}
\begin{proof}
    By Corollary \ref{red-to-CI}, Lemma \ref{red-to-CI}, Lemma \ref{F-fixe} and Lemma \ref{red-to-BK}, we have $H_c^*(\tSigma_w, \ov\BQ_\ell)[\tphi] = 0$ unless $w \in \Theta_{\tT}^F$ and $\tth_w(\Phi_{G^i}) = \Phi_{G^i}$ for $0 \le i \le d$, in which case $H_c^*(\tSigma_w, \ov\BQ_\ell)[\tphi] = H_c^*(\Sigma_w^\tBK, \ov\BQ_\ell)[\tphi]$. Note that the connected group $T_{0+:r}^{\tth_w} = F T_{0+:r}^{\tth_w}$ acts on $\Sigma_w^\tBK$ by $t: (x_0, y_0, z, \t, \psi) \mapsto (x_0, y_0, z, t\t, \psi)$, which extends the action of $(T_{0+:r}^{\tth_w})^F$ on $\Sigma_w^\tBK$. Hence $H_c^*(\Sigma_w^\tBK, \ov\BQ_\ell)[\tphi] = 0$ unless $\tphi$ is trivial over $(T_{0+:r}^{\tth_w})^F$, that is, $w \in \Theta_{\tT, \tphi, +}^F$ as desired.
\end{proof} 

\subsection{The varieties $\Sigma_w^\tBL$ and $\Xi_w^\BH$}
Let $\bar \BL_U$ and $\bar \CO_w$ be the natural images of $\BL_U$ and $\tCO_w$ respectively under the natural projection $\pr: \tBK \to \tBK / \BH$. Let \begin{align*} \Sigma_w^\tBL &= \{(x, y, \psi) \in F\BL_U \times \BL_U \times \tCO_w; F(\psi) =  x y \psi\}; \\ \Sigma_w^{\tL_0} &= \{(x, y, \psi) \in F\bar\BL_U \times \bar\BL_U \times \bar\CO_w; F(\psi) =  x y \psi\}.\end{align*}
\begin{lemma} \label{lift-iso}
    The projection induces an isomorphism $\Sigma_w^\tBL \cong \Sigma_w^{\tL_0}$.
\end{lemma}
\begin{proof}
    Note that the natural projections $\BL_U \to \bar \BL_U$, $F\BL_U \to F\bar \BL_U$ and $\tCO_w \to \bar \CO_w$ are isomorphisms. Hence the projection $\Sigma_w^\tBL \to \Sigma_w^{\tL_0}$ is injective. Let $(x, y, \psi) \in F\BL_U \times \BL_U \times \tCO_w$ whose image under $\pr$ belongs to $\Sigma_w^{\tL_0}$. It remains to show $(x, y, \psi) \in \Sigma_w^\tBL$. Suppose $\psi = s \tBK^{\tth_w}$ for some $s \in \tT_0$. By definition we have \[F(s)\i x y s = t h \in \BH \tBK^{\tth_w},\] where $t = F(s)\i s \in \tT_0$ and $h = s\i xy s \in \tBL_\der$. Here $\tBL_\der$ denotes the natural image of $(\tL_\der)_r$ in $\tBK$. As a consequence, \[\tth_w(th) (th)\i =  a b\in \BH,\] where $a = (\tth_w(t)t\i) \in \tT_0$ and $b = {}^t(\tth_w(h) h\i) \in \tBL_\red$. Hence $b \in \tBT \cap \tBL_\der$. Let $b = b_s b_u$ be the Jordan decomposition with $b_s$ simisimple and $b_u$ unipotent. Then $b_u \in \BT^+ \cap \tBL_\der$ is trivial and hence $a b = a b_s \in \BH \cap \tT_0$, which implies that $ab = a b_s = 1$. Hence $t h \in \tBK^{\tth_w}$ and $(u, y, \psi) \in \Sigma_w^\tBL$ as desired.    
\end{proof}

Let $\Xi_w^\BH = \{(z, \t, \psi) \in  F\BH_U \BH_U \times \BT^+ \times \tCO_w; \psi = L(\t) z \psi\}$.
\begin{proposition} \label{fiber-product}
    The map $(x_0, y_0, z, \t, \psi) \mapsto ((x_0, y_0, \psi), (z, \t, \psi))$ gives a $T_r^F$-equivariant isomorphism $\Sigma_w^\tBK \cong \Sigma_w^{\tL_0} \times_{\tCO_w} \Xi_w^\BH$.
\end{proposition}
\begin{proof}
    It suffices to show map is well-defined. Let $(x_0, y_0, z, \t, \psi) \in \Sigma_w^\tBK$. By definition $\pr(x_0, y_0, \psi) \in \Sigma_w^{\tL_0}$. By Lemma \ref{lift-iso} we have $(x_0, y_0, \psi) \in \Sigma_w^\tBL $ and hence $(z, \t, \psi) \in \Xi_w^\BH$ as desired. 
\end{proof}

\section{The cohomology of $\Sigma_w^\BK$} \label{sec:coho-BK}
In this section, we compute the cohomology of $\Sigma_w^\BK$. Let notation be as in \S \ref{sec:first-reduction}. Let $\hCO_w = \tT_0 / \tT_0^{\tth_w, \circ}$ and $\k: \hCO_w \to \tT_0 / \tT_0^{\tth_w} \cong \tCO_w$ the natural projection. Let $\G_w = T_0^{\tth_w} / T_0^{\tth_w, \circ}$. 

Consider the following varieties \begin{align*}\hSigma_w^{\tL_0} &= \{(x, y, \hpsi) \in F\BL_U \times \BL_U \times \hCO_w; F(\k(\hpsi)) = x y \k(\hpsi)\}; \\ \hXi_w^\BH &= \{(z, \t, \hpsi) \in  F\BH_U \BH_U \times \BT^+ \times \hCO_w; \k(\hpsi) = L(\t) z \k(\hpsi)\} \end{align*} We put $\hSigma_w^\tBK = \hSigma_w^{\tL_0} \times_{\hCO_w} \hXi_w^\BH$.
\begin{lemma}\label{torsor}
    The natural map $\k: \hCO_w \to \tCO_w$ induces a $\G_w$-torsor $\hSigma_w^\tBK \to \Sigma_w^\tBK$.
\end{lemma}
\begin{proof}
    It follows from Proposition \ref{fiber-product}.
\end{proof}

\subsection{The cohomolgy of $\hat\Sigma_w^\tBK$}
As $\BH_U$ is commutative, we can write \[F\BH_U \BH_U = \prod_{\a \in F\Phi^+ \cup \Phi^+} \BH_\a,\] where the product is taken with respect to a fixed total order $\preceq$ on $\Phi$ such that $\a \preceq \b$ if $\b - \a$ is a sum of positive roots. 

Let $\Psi_w \subseteq F\Phi^+ \cup \Phi^+$ be the set of roots $\a$ such that either $\a \prec \tth_w(\a) \in F\Phi^+ \cup \Phi^+$ or $\a = \tth(\a)$ and $\tth_w$ acts on $\BH_\a$ trivially. Define \begin{align*} \Psi_w' &= \{\a \in \Psi_w \cap \Phi^-; \a \preceq -\tth_w(\a) \in \Psi_w \cap \Phi^-\}; \\ \Psi_w'' &= \Psi_w' \cap -\tth_w(\Psi_w') = \{\a \in \Psi_w'; \a = -\tth_w(\a)\}. \end{align*}
For $D \subseteq \Phi$ we put $\BH_D = \prod_{\a \in D} \BH_\a$. We define a morphism \[\chi: \hCO_w \times \BH_{\Psi_w} \mapsto \BT^+, \quad (\hpsi, (z_\a)_\a) \mapsto  \sum_{\a \in \Psi_w'} [z_\a,  {}^{\tth_\hpsi} z_{-\tth_w(\a)}].\]

The following result follows directly from the commutator relations of $\BH$ in \S\ref{subsec:BK}.
\begin{lemma} \label{cartesian}
    There is a Cartesian diagram \[ \xymatrix{
    \hXi_w^\BH \ar[d]_{p_{13}} \ar[r]^{p_2} &  \BT^+ \ar[d]_{\bar L} \\
     \hCO_w \times \BH_{\Psi_w} \ar[r]^{\chi} & \BT^+/(\BT^+)^{\tth_w}.}\] Here $p_{13}$, $p_2$ are natural projection maps, and $\bar L$ is the composition of the Lang's map $L$ with the natural projection $\BT^+ \to \BT^+/(\BT^+)^{\tth_w}$.
\end{lemma}

Let $\CL_+$ be the rank one multiplicative local system over $\BT^+$ corresponding to the character $\tphi |_{(\BT^+)^F}$, which is still denoted by $\tphi^+$.

\begin{proposition} \label{criterion}
    Let $f: X \times \BG_a \to \BT^+$ be a morphism. Let $\CL$ be a local system over $\BT^+$. Suppose that for each $x \in X$ the pull-back of $\CL$ via the map $f_x: z \mapsto f(x, z)$ is isomorphic to a rank one multiplicative local system on $\BG_a$. Then \[R\Gamma_c(X \times \BG_a, f^*\CL) \cong R\Gamma_c(X_0 \times \BG_a, f^*\CL),\] where $X_0$ is the set of points $x \in X$ such that $f_x^*\CL$ is the trivial multiplicative local system.
\end{proposition}

\begin{proposition} \label{sheaf}
    We have a natural $\G_w$-equivariant isomorphism \[R\G_c(\hSigma_w^\tBK, \ov\BQ_\ell)[\tphi^+] \cong R\G_c(\hSigma_w^{\tL_0} \times \BH_{\Psi_w''}, h^*\CL_+)[2|\Psi_w \sm \Psi_w'| + 2\dim (\BT^+)^{\tth_w}].\] Here $h: \hSigma_w^{\tL_0} \times \BH_{\Psi_w''} \to \BT^+$ is given by $(x,  y, \hpsi, z) \mapsto \chi(\hpsi, z)$. 
\end{proposition}
\begin{proof}
    By Proposition \ref{cartesian}, there is a Cartesian diagram \[ \xymatrix{
    \hSigma_w^\tBK \ar[d]_{q} \ar[r]^p &  \BT^+ \ar[d]_{\bar L} \\
     \hSigma_w^{\tL_0} \times \BH_{\Psi_w} \ar[r]^h & \BT^+/(\BT^+)^{\tth_w},}\] where the maps $p$ and $q$ send $(x, y, z, \t, \psi)$ to $\t$ and to $(\bar x, \bar y, \psi, z)$ respectively. 

     Note that $(\BT^+)^{\tth_w}$ is an affine space. By base change theorem we have \[q_! \ov\BQ_\ell = q_! p^* \ov\BQ_\ell \cong h^* \bar L_! \ov\BQ_\ell \cong \oplus_{\l} h^* \k_! \CL_\l \cong \oplus_{\bar \l} h^*\CL_{\bar \l}[2\dim (\BT^+)^{\tth_w}],\] where $\k: \BT^+ \to \BT^+/(\BT^+)^{\tth_w}$, $\l$ (resp. $\bar \l$) ranges over characters of $(\BT^+)^F$ (resp. $(\BT^+/(\BT^+)^{\tth_w})^F$) and $\CL_\l$ (resp. $\CL_{\bar \l}$) is the corresponding rank one multiplicative local system. Then we have \[R\G_c(\hSigma_w^\tBK, \ov\BQ_\ell)[\tphi^+] \cong R\G_c(\hSigma_w^{\tL_0} \times \BH_{\Psi_w}, h^*\CL_+)[2\dim (\BT^+)^{\tth_w}].\] Applying Proposition \ref{criterion} we have \[R\G_c(\hSigma_w^{\tL_0} \times \BH_{\Psi_w}, h^*\CL_+) \cong R\G_c(\hSigma_w^{\tL_0} \times \BH_{(\Psi_w \sm \Psi_w') \sqcup \Psi_w''}, h^*\CL_+).\] Note that the restriction of $h$ to $\hSigma_w^{\tL_0} \times \BH_{(\Psi \sm \Psi_w) \sqcup \Psi_w''}$ factors through the natural projection $\hSigma_w^{\tL_0} \times \BH_{(\Psi_w \sm \Psi_w') \sqcup \Psi_w''} \to \hSigma_w^{\tL_0} \times \BH_{\Psi_w''}$. Using the projection formula we have \[ R\G_c(\hSigma_w^{\tL_0} \times \BH_{(\Psi_w \sm \Psi_w') \sqcup \Psi_w''}, h^*\CL_+) \cong  R\G_c(\hSigma_w^{\tL_0} \times \BH_{\Psi_w''}, h^*\CL_+)[2|\Psi_w \sm \Psi_w'|].\] Hence the statement follows.    
\end{proof}

\subsection{The cohomology of $\Sigma_w^\tBK$}
Let $\a \in \Psi_w''$. Then $\a = -\tth_w(\a)$, and hence $\a$ is trivial over $\tT_0^{\tth_w, \circ}$. For $\hpsi \in \tT_0/\tT_0^{\tth_w}$ we put $\a(\hpsi) = \a(t)$ for some/any lift $t \in \tT_0$ of $\hpsi$. Moreover, $\a(\hpsi) \in \{\pm 1\}$ for $\hpsi \in \G_w$.

We define an automorphism \[\nu: \hSigma_w^{\tL_0} \times \BH_{\Psi_w''} \overset \sim \to \hSigma_w^{\tL_0} \times \BH_{\Psi_w''}, \quad (x, y, \hpsi, (z_\a)_\a) \mapsto (x, y, \hpsi, (\a(\hpsi)z_\a)_\a).\] 
\begin{lemma} \label{split}
    We have $h \circ \nu = p_1 \circ c$. Here $p_1: \hSigma_w^{\tL_0} \times \BH_{\Psi_w''} \to \BH_{\Psi_w''}$ is the natural projection map, and $c: \BH_{\Psi_w''} \to \BT^+$ is given by \[(z_\a)_\a  \mapsto \sum_{\a \in \Psi_w''} \a^\vee(1 + \varpi^{r_\a} d_\a z_\a^2),\] where $d_\a \in \ov\BF_q^\times$ are certain non-zero constants.
\end{lemma}

\begin{proposition} \label{square}
    Let $f: \BG_a \to  \BG_a$ be given by $x \mapsto x^2$. Let $\CL$ be a non-trivial multiplicative sheaf over $\BG_a$. Then $R\G_c(\BG_a, f^*\CL)$ concentrates at degree one with dimension one.
\end{proposition}

\begin{proposition} \label{red-to-zero}
    We have $\dim H_c^*(\Sigma_w^\tBK, \ov\BQ_\ell)[\tphi^+] = (-1)^{|\Psi_w''|} \dim H_c^*(\Sigma_w^{\tL_0}, \ov\BQ_\ell)$.
\end{proposition}
\begin{proof}
    Let $m = |\Psi_w \sm \Psi_w'|$. By Lemma \ref{torsor} and Proposition \ref{sheaf}, we have \begin{align*}
        &\quad\ R\G(\Sigma_w^\tBK, \ov\BQ_\ell)[\tphi^+] \\ &\cong R\G(\hSigma_w^\tBK / \G_w, \ov\BQ_\ell)[\tphi^+] \\ &\cong (R\G(\hSigma_w^\tBK, \ov\BQ_\ell)[\tphi^+])^{\G_w} \\ &\cong R\G_c(\hSigma_w^{\tL_0} \times \BH_{\Psi_w''}, h^*\CL_+)^{\G_w}[2m] \\ &\cong R\G_c(\hSigma_w^{\tL_0} \times \BH_{\Psi_w''}, (h \circ \nu)^*\CL_+)^{\G_w}[2m] \\ &\cong R\G_c(\hSigma_w^{\tL_0} \times \BH_{\Psi_w''}, (p_1 \circ c)^*\CL_+)^{\G_w}[2m] \\ &\cong (R\G_c(\hSigma_w^{\tL_0}, \ov\BQ_\ell) \otimes R\G_c(\BH_{\Psi_w''}, c^*\CL_+))^{\G_w}[2m],
    \end{align*} where, by tracing the isomorphism $\nu$, the $\G_w$-actions on $R\G_c(\hSigma_w^{\tL_0}, \ov\BQ_\ell)$ and $R\G_c(\BH_{\Psi_w''}, c^*\CL_+)$ in the bottom line are induced from the $\G_w$-actions on $\hSigma_w^{\tL_0}$ and $\BH_{\Psi_w''}$ given by $\g: (x, y, \hpsi) \mapsto (x, y, \g\hpsi)$ and $\g: (z_\a)_\a \mapsto (\a(\g) z_\a)_\a$ respectively. 
    
    As $\hSigma_w^{\tL_0}$ is a trivial $\G_w$-torsor over $\Sigma_w^{\tL_0}$, there is an isomorphism of $\G_w$-modules \[R\G_c(\hSigma_w^{\tL_0}, \ov\BQ_\ell) \cong R\G_c(\Sigma_w^{\tL_0}, \ov\BQ_\ell) \otimes \ov\BQ_\ell \G_w \cong \bigoplus_\l R\G_c(\Sigma_w^{\tL_0}, \ov\BQ_\ell) \otimes \l,\] where $\l$ ranges over characters of $\G_w$.

    Now we compute the $\G_w$-module $R\G_c(\BH_{\Psi_w''}, c^*\CL_+)$. As $\CL_+$ is a multiplicative local system, it follows from Lemma \ref{split} that \[R\G_c(\BH_{\Psi_w''}, c^*\CL_+) \cong \bigotimes_{\a \in \Psi_w''} R\G_c(\BH_\a, c^*\CL_+).\] Applying Proposition \ref{square}, each $R\G_c(\BH_\a, c^*\CL_+)$ concentrates at degree one with dimension one. Hence $R\G_c(\BH_{\Psi_w''}, c^*\CL_+)$ concentrates at degree $|\Psi_w''|$ with dimension one. In particular, $\G_w$ acts on $R\G_c(\BH_{\Psi_w''}, c^*\CL_+)$ via a character $\l_w$. Therefore, \[(R\G_c(\hSigma_w^{\tL_0}, \ov\BQ_\ell) \otimes R\G_c(\BH_{\Psi_w''}, c^*\CL_+))^{\G_w} \cong R\G_c(\Sigma_w^{\tL_0}, \ov\BQ_\ell)[|\Psi_w''|]\] and the statement follows.    
\end{proof}

\section{Proof of Theorem \ref{tilde-formula}} \label{sec:proof-tilde}

In this section, we finish the proof of the main result. First we recall Lusztig's result in the $r=0$ case.
\begin{theorem} [{\cite[Proposition 6.3]{Lu90}}] \label{Euler-zero}
    For $t \in \tT_0^F$ we have \[H_c^*((\Sigma_w^{\tL_0})^t, \ov\BQ_\ell) = (-1)^{|\Phi_w''(t)|} |\tT_0^F| |(\tT_0^{\tth_w, \circ})^F|\i.\] Here $\Phi_w''(t) = \{\a \in \Phi_{G^0}^+; \a(t) = 1, \tth_w(\a) = -\a, F\i(\a) < 0\}$
\end{theorem}

\begin{lemma} \label{orbit}
    For $w \in \Theta_{\tT_r}^F$ we have \[|\tT_r^F \backslash (\tT_r w \tG_r^{\tth})^F / (\tG_r^{\tth})^F| = |(\tT_0^{\tth_w})^F| \cdot |(\tT_0^{\tth_w, \circ})^F|\i.\]
\end{lemma}
\begin{proof}
    As $\tG_r^{\tth}$ is connected, by \cite[Lemma 7.4]{Lu90} we have \begin{align*}  &\quad\ |\tT_r^F \backslash (\tT_r w \tG_r^{\tth})^F / (\tG_r^{\tth})^F| = |\tT_r^F \backslash (\tT_r w \tG_r^{\tth} / \tG_r^{\tth})^F| \\ &=  |(\tT_r^{\tth_w})^F / (\tT_r^{\tth_w, \circ})^F| =  |(\tT_0^{\tth_w} / \tT_0^{\tth_w, \circ})^F| = |(\tT_0^{\tth_w})^F| \cdot |(\tT_0^{\tth_w, \circ})^F|\i.\end{align*}
\end{proof}

\begin{lemma} \label{symm}
    For $w \in \Theta_{\tT_r, \tphi, +}^F$ we have $\tth_w(\Phi_{G^i}) = \Phi_{G^i}$ for $0\le i \le d$.
\end{lemma}
\begin{proof}
    By definition, $\tphi$ is trivial over $(T_{0+:r}^{\tth_w})^F$. In particular, $\tphi(\t \tth_w(\t)) = 1$ for all $\t \in T_{0+:r}^F$, that is, $\phi_+ = \tphi_+ = \tth_w(\tphi_+\i) = \tth(\phi_+\i)$. Then the statement follow from the uniqueness of the Levi subgroups sequence in the Howe factorizations of $\phi_+$ and $\phi_+\i$, which coincide with each other by the construction in \cite[\S 3.6]{Kal}.
\end{proof}

\begin{proposition} \label{coh-BK}
    Let $w \in \Theta_{\tT_r}$. We have $\dim H_c^*(\tSigma_w, \ov\BQ_\ell)[\tphi] = 0$ unless $w \in \Theta_{\tT_r, \tphi}^F$, in which case it equals $(-1)^{n_w} |\tT_r^F \backslash (\tT_r w \tG_r^{\tth})^F / (\tG_r^{\tth})^F|$.
\end{proposition}
\begin{proof}
     By Lemma \ref{trivial} and Lemma \ref{symm}, we have $\dim H_c^*(\tSigma_w, \ov\BQ_\ell)[\tphi] = 0$ unless $w \in \Theta_{\tT_r, \tphi, +}^F$, in which case $\dim H_c^*(\tSigma_w, \ov\BQ_\ell)[\tphi] = \dim H_c^*(\Sigma_w^\tBK, \ov\BQ_\ell)[\tphi]$. Now we assume $w \in \Theta_{\tT_r, \tphi, +}^F$, that is, $\tphi|_{(\tT_r^{\tth_w})^F}$ factors through $\tphi|_{(\tT_0^{\tth_w})^F}$. Let $t \in \tT_0^F$. We denote by \[(\Sigma_w^\tBK)^t = \{(x, y, z, \t, \psi) \in \Sigma_w^\tBK; txt\i=x, tyt\i=y, tzt\i=z, t\psi =\psi\}\] the set of $t$-fixed points in $\Sigma_w^\tBK$. Let $\Psi_w''(t) = \{\a \in \Psi_w; \a(t) = 1\}$. Then we have
    \begin{align*}
        &\quad \ \dim H_c^*(\tSigma_w, \ov\BQ_\ell)[\tphi] \\ &=\dim H_c^*(\Sigma_w^\tBK, \ov\BQ_\ell)[\tphi]  \\ &= |\tT_r^F|\i\sum_{t \in \tT_r^F} \tphi(t\i) \tr(t; H_c^*(\Sigma_w^\tBK, \ov\BQ_\ell)) \\ &= |\tT_0^F|\i \sum_{t_s \in \tT_0^F} \tphi(t_s\i) |\tT_{0+:r}^F|\i \sum_{t_u \in \tT_{0+:r}^F} \tphi(t_u\i) \tr(t_u; H_c^*((\Sigma_w^\tBK)^{t_s}, \ov\BQ_\ell)) \\  &= |\tT_0^F|\i \sum_{t_s \in \tT_0^F} \tphi(t_s\i) \dim H_c^*((\Sigma_w^\tBK)^{t_s}, \ov\BQ_\ell)[\tphi^+] \\  &= |\tT_0^F|\i \sum_{t \in \tT_0^F} \tphi(t\i) (-1)^{|\Psi_w''(t)|} H_c^*((\Sigma_w^{\tL_0})^t, \ov\BQ_\ell) \\ &= |\tT_0^F|\i \sum_{t \in (\tT_0^{\tth_w})^F} \tphi(t\i) (-1)^{|\Psi_w''(t)|} H_c^*((\Sigma_w^{\tL_0})^t, \ov\BQ_\ell) \\ &= |\tT_0^F|\i \sum_{t \in (\tT_0^{\tth_w})^F} \tphi(t\i) (-1)^{|\Psi_w''(t)| + |\Phi_w''(t)|} |\tT_0^F| |(\tT_0^{\tth_w, \circ})^F|\i  \\ &= |(\tT_0^{\tth_w, \circ})^F|\i (-1)^{n_w} \sum_{t \in (\tT_0^{\tth_w})^F} \tphi(t\i) \e_w(t), \\ &=\begin{cases}
            (-1)^{n_w} |(\tT_0^{\tth_w})^F| \cdot |(\tT_0^{\tth_w, \circ})^F|\i, &\text{ if } \e_w = \tphi|_{(\tT_0^{\tth_w})^F}; \\
            0, &\text{ otherwise, }
        \end{cases}
    \end{align*} where the fifth equality follows from Proposition \ref{red-to-zero} (applied to $(\Sigma_w^\tBK)^t$); the sixth equality follows that $(\Sigma_w^{\tL_0})^t = \emptyset$ if $t \notin (\tT_0)^{\tth_w}$; the seventh equality follows from Theorem \ref{Euler-zero}; the eighth equality follows from the equality $n_w(t) = |\Phi_w''(t)| + |\Psi_w''(t)|$. The statement then follows from Lemma \ref{orbit}.
\end{proof}

Now we are ready to finishe the proof of Theorem \ref{tilde-formula}.
\begin{proof}
By Proposition \ref{coh-BK} we have
\begin{align*} &\quad\ \frac{1}{|(\tG_r^\tth)^F|}\sum_{g \in (\tG_r^\tth)^F} \tr(g; R_{\tT_r}^{\tG_r}(\tphi)) \\ &=\dim \Hom_{\tG_r^\tth)^F}(R_{\tT_r}^{\tG_r}(\tphi), 1) \\ &= \sum_{w \in \tT_r \backslash \Theta_{\tT_r} / \tG_r^\tth} \dim H_c^*(\tSigma_w, \ov\BQ_\ell)[\tphi] \\ &= \sum_{w \in \tT_r \backslash \Theta_{\tT_r} / \tG_r^\tth, \ w \in  \Theta_{\tT_r, \tphi}^F} \dim H_c^*(\tSigma_w, \ov\BQ_\ell)[\tphi] \\ &= \sum_{w \in \tT_r \backslash \Theta_{\tT_r} / \tG_r^\tth, \ w \in  \Theta_{\tT_r, \tphi}^F} (-1)^{n_w} |\tT_r^F \backslash (\tT_r w \tG_r^{\tth})^F / (\tG_r^{\tth})^F| \\  &= \sum_{w \in  \tT_r^F \backslash \Theta_{\tT_r, \tphi}^F / (\tG_r^{\tth})^F} (-1)^{n_w}.\end{align*}  So the proof is finished.
\end{proof}

\end{document}